\documentclass[12pt]{amsart}
\usepackage{newlfont,amsmath,enumerate,amssymb,array}
\usepackage{newlfont,amsmath,mathrsfs,enumerate,amssymb,array,amscd,xcolor}

\newcommand{\calF}{{\mathcal{F}}}

\newcommand{\bbA}{{\mathbb A}}

\newcommand{\bbC}{{\mathbb C}}

\newcommand{\bbR}{{\mathbb R}}

\newcommand{\Aut}{{\mathrm{Aut}}}

\newcommand{\der}{{\mathrm{der}}}

\newcommand{\Spin}{{\mathrm{Spin}}}
\newcommand{\SL}{{\mathrm{SL}}}
\newcommand{\GL}{{\mathrm{GL}}}

\newcommand{\Sp}{\mathrm{Sp}}
\newcommand{\Hom}{{\mathrm{Hom}}}
\newtheorem{lemma}{Lemma}[section]

\newtheorem{prop}[lemma]{Proposition}
\newtheorem{thm}[lemma]{Theorem}
\newtheorem{cor}[lemma]{Corollary}
\newtheorem{definition}[lemma]{Definition}

\begin{document}


\title{Global uniqueness of small representations}

\author{Toshiyuki Kobayashi and Gordan Savin}
 \address{Toshiyuki Kobayashi, Kavli IPMU and Graduate School of Mathematical Sciences, the University of Tokyo}
 \email{toshi@ms.u-tokyo.ac.jp}

\address{Gordan Savin, Department of Mathematics, University of Utah,
  Salt Lake City, UT 84112} 
  \email{savin@math.utah.edu}

\begin{abstract} 
We prove that  automorphic representations whose local components are certain small representations have multiplicity one. The proof is based on 
the multiplicity-one theorem for certain functionals of small representations,  also proved in this paper. 

\end{abstract} 


\maketitle

\section{Introduction} 

Let $k$ be a field of characteristic 0. Let $G$ be the group of $k$-points of a simply connected, absolutely simple algebraic group defined over $k$, with 
the Lie algebra $\mathfrak g$.   Let  $P=MN$ be a maximal parabolic subgroup with abelian unipotent radical $N$  such that $P$ is conjugate to 
the opposite parabolic subgroup $\bar P =M \bar N$ by an element in $G$.  In this case $N$ and $\bar N$ admit a structure of 
simple Jordan algebra $J$.  The Jordan algebra structure sheds light on the structure of $M$-orbits on $\bar N$. More precisely, we have a 
decomposition 
\[ 
\bar N = \coprod_{j=0}^r \Omega_j 
\] 
where $\Omega_j$ is the set of elements of ``rank $j$'' and $r$ the degree of $J$.  
A precise definition is given in Section 4, but the reader is probably familiar with the following example. 
If $G=\Sp_{2r}(k)$ and $P$ is the Siegel maximal parabolic subgroup, then $\bar N$ can be identified with the space of $r\times r$ symmetric matrices, and $\Omega_j$ 
consist of all symmetric matrices of rank $j$. Over an algebraically closed field, $M$ acts transitively on every $\Omega_j$. 
Over a general field $k$,  however, the structure of $M$-orbits may be complicated. 
 
  If $k$ is a local field,  then $\bar N$ can be identified with the Pontrjagin dual of $N$. In particular, any $x\in \bar N$ corresponds to 
a unitary character $\psi_x : N \rightarrow \mathbb C^{\times}$. Let $\omega\subseteq \Omega_j$ be an $M$-orbit where $j<r$.
 We have an irreducible representation $\pi$ of $P$ 
on the Hilbert space $\mathcal H=L^2(\omega)$,
  defined with respect to a quasi $M$-invariant measure on $\omega$. The action of $M$ on $L^2(\omega)$ 
 arises from the geometric action of $M$  on $\omega$, while $n\in N$ acts on $f\in L^2(\omega)$ by 
\[ 
\pi(n) f(y)= \psi_y(n) f(y).  
\] 
The small representations in the title of this work are unitary representations of $G$ whose restriction to $P$ is isomorphic to $(\pi, L^2(\omega))$ for some 
$\omega$.  If $G=\Sp_{2r}(k)$, then small representations appear  naturally in the stable range of theta correspondences, see \cite{Ho}. 
For more general $G$, we have works of \cite{Sa},  \cite{hkm}, \cite{km}, for real groups, and works of \cite{to} and \cite{mw} for $p$-adic groups.

Let $\mathcal H^{\infty}$ be the space of 
$G$-smooth vectors in $\mathcal H=L^2(\omega)$. Since $M$ acts transitively on $\omega$, elements in $\mathcal H^{\infty}$ are represented by smooth functions on 
$\omega$. In particular, we can evaluate  $f\in \mathcal H^{\infty}$ at any point $x\in \omega$. The functional 
\[ 
\delta_x:{\mathcal{H}}^{\infty} \to {\mathbb{C}}, 
\quad
f\mapsto f(x) 
\] 
is  continuous on $\mathcal H^{\infty}$ and  $(N,\psi_x)$-equivariant 
{\it{i.e.}} for all $n\in N$ and $f\in \mathcal H^{\infty}$ 
\[ 
\delta_x(\pi(n) f) = \psi_x(n) \delta_x(f). 
\] 
One may ask if any  continuous and $(N,\psi_x)$-equivariant functional $\ell$ is a multiple of $\delta_x$. We show that this is indeed the case (Propositions \ref{P:p-adic} and \ref{P:real}), 
under a natural assumption that 
$\mathfrak g$ acts on $\mathcal H^{\infty}$ by regular differential operators on $\omega$ if $k$ is archimedean. 

We now explain the key points of the proof. It is not too difficult to see that $\ell(f)=0$ for any function $f$ vanishing in a neighborhood of $x$. 
If $k$ is a $p$-adic field and $f_1(x)=f_2(x)$, for a pair of smooth functions, then 
the difference $f_1-f_2$ vanishes in a neighborhood of $x$. Hence $\ell(f_1)=\ell(f_2)$ and this implies that $\ell$ is a multiple of $\delta_x$. However, if $k=\mathbb R$, then 
$f_1(x)=f_2(x)$, for a pair of smooth functions,  does not imply that $f_1-f_2$ vanishes in a neighborhood of $x$.
 Moreover, a priori, it is not clear that $\mathcal H^{\infty}$ contains a single non-zero function vanishing in a neighborhood of $x$. 
 For example, $K$-finite elements in $\mathcal H$ are represented by analytic functions on $\omega$. 
In order to prove that $\ell$ is a multiple of $\delta_x$ we first prove that 
$C^{\infty}_c(\omega)$, the space of smooth compactly supported functions on $\omega$,  is contained in $\mathcal H^{\infty}$ and then we reduce the problem to some standard results in the theory of distributions. 
A key in proving that  $C^{\infty}_c(\omega)\subseteq  \mathcal H^{\infty}$ is the following analogue of the Sobolev lemma, a general result of independent interest. 
Let $(\pi, \mathcal H)$ be any unitary representation of $G$. Let $v\in \mathcal H$. It defines a continuous functional on $\mathcal H^{\infty}$,
 by the inner product 
on $\mathcal H$. 
 The enveloping algebra $U(\mathfrak g)$ of 
$\mathfrak g$ acts on $\mathcal H^{\infty}$ and hence we have a  weak $d\pi^{-\infty}$ action of $U(\mathfrak g)$ on $v$. If, for all $u\in U(\mathfrak g)$, 
$d\pi^{-\infty}(u) v$ is in $\mathcal H$, then $v$ is $G$-smooth. We remind the reader that the classical  Sobolev lemma states that  if all weak derivatives 
({\it{i.e.}} derivatives in the sense of distributions) of $f\in L^2(\mathbb R)$ are 
contained in $L^2(\mathbb R)$, then $f$ is a smooth function. The analogy is obvious. 

Next, following Howe \cite{Ho}, we define a notion of $N$-rank for smooth representations of $G$. A smooth representation $\pi$ has $N$-rank $j$ if 
there exists a non-zero, continuous, $(N,\psi_y)$-equivariant functional on $\pi$ for some $y\in \Omega_j$, but there is no such functional for $y$ in larger orbits. 
In particular, the previous discussion can be summarized as follows.  A small representation has the $N$-rank $j$ where $j$ is the integer such that $\omega\subseteq \Omega_j$ and, for every $y\in \omega$, 
any  $(N,\psi_y)$-equivariant functional is a multiple of $\delta_y$.  We use this information to show that automorphic representations whose  
local components are small have multiplicity one. 
The following is the main result of this paper. It is a combination of  Theorems \ref{global} and \ref{T:unique}.

\begin{thm} 
 Let $\mathbb A$ be the ring of adel\'es  corresponding to an algebraic number field. 
 Let $\pi =\hat\otimes_{v} \pi_v$ be an automorphic representation of $G(\mathbb A)$ such that 
$\pi_v$ is a small representation for every place $v$. Then the $N$-rank of $\pi_v$ is independent of $v$ and 
 $\pi$ has multiplicity one in the automorphic spectrum. 
\end{thm}

The paper is organized as follows. Sections 2 and 3  contain a precise description of groups considered in this paper. Starting with a split group $G$, 
  we define a structure of simple Jordan algebra $J$ on $N$ and $\bar N$. We show that there is a natural inclusion of groups 
\[ 
\Aut(J) \hookrightarrow \Aut(G)=\Aut(\mathfrak g). 
\] 
Thus, any class $c$ in $H^1(k, \Aut(J))$ defines a form $J^c$ of $J$ and a form $G^c$ of $G$, containing a maximal parabolic subgroup $P^c$ whose nilpotent radical $N^c$ has a structure of the Jordan algebra $J^c$.  This is the Kantor--Koecher--Tits construction \cite{ja}, page 324, from the Galois cohomology point of view. 
In Section 4 we discuss the Hasse principle for $M$-orbits on $N$. Section 6 contains the analogue of the Sobolev lemma, described above. 
Sections 7 and 8 contain proofs of the uniqueness of $(N,\psi_x)$-equivariant functionals for $x\in \omega$ in the $p$-adic and real cases, 
respectively. In Section 9 we define the notion of $N$-rank for representations of $G$ and prove that the local components of an automorphic 
representation have the same $N$-rank.  In Section 10 we prove the global multiplicity one statement. In particular, we prove that the minimal representations appear in the automorphic spectrum with multiplicity one (Corollary \ref{C:min}). 

\section{Jordan algebras} \label{preliminaries}

Let $G$ be as in the introduction.  The main purpose of this section is to explain the Jordan structure on $N$ and $\bar N$. 
We shall do this first for split groups. A more general case will be treated in the next section using Galois descent. 

So we assume that $G$ is split throughout this section.  Fix 
$\mathfrak t\subseteq \mathfrak g$, a maximal split Cartan subalgebra.
Let $\Phi$ be the root system
 for $({\mathfrak {g}}, {\mathfrak {t}})$ and, for every $\alpha\in\Phi$, let 
$\mathfrak g_{\alpha}\subseteq \mathfrak g$ be the corresponding root space. 
 Fix $\Delta=\{\alpha_{1}, \ldots ,\alpha_{l}\}$, a set of simple roots. Now every root can be written as
a sum $\alpha= \sum_{i=0}^l m_i(\alpha)\alpha_i$ for some integers
$m_i(\alpha)$. Every simple root  $\alpha_{j}$ defines a maximal parabolic subalgebra 
$\mathfrak p \equiv \mathfrak p_j =\mathfrak m + \mathfrak n$  where 
\begin{align*}
\mathfrak m =& \mathfrak t \oplus (\bigoplus _{m_{j}(\alpha)=0}~ \mathfrak g_{\alpha}), 
\\
\mathfrak n =& \bigoplus_{m_{j}(\alpha)>0}~\mathfrak g_{\alpha}.
\end{align*}
Note that $\mathfrak m_{\der} =[\mathfrak m,\mathfrak m]$ is a semi-simple Lie algebra
which corresponds to the Dynkin diagram of 
$\Delta \setminus \{\alpha_{j}\}$. Let $\beta$ be the highest root. The algebra
 $\mathfrak n$ is commutative if and only if $m_{j}(\beta)=1$. 
  Here is the list of all possible pairs 
$(\mathfrak g, \mathfrak m)$ with $\mathfrak n$ commutative and $\mathfrak p$ conjugate to the opposite parabolic by an element in $G$. 
\[
 \begin{array}{c||c|c|c|c|c|c}
 \mathfrak g  & C_n &A_{2n-1} &  D_{2n}  & E_{7} & B_{n+1} & D_{n+1}  \\ 
 \hline
   \mathfrak m_{\der} &A_{n-1} &  A_{n-1}\times A_{n-1} &   A_{2n-1} & E_{6} & B_n &  D_{n}  \\
   \hline 
   \dim \mathfrak n  & n(n+1)/2 &  n^2 &  n(2n-1) & 27 & 2n+1 & 2n  \\
   \hline 
   r & n & n & n & 3 & 2 & 2 \\
   \hline 
   d & 1 & 2 & 4 & 8 & 2n-1 & 2n-2  \\
  \end{array}  
 \] 
 
 The meaning of the integer $d$ will be explained later. The integer 
  $r$ is the cardinality of any maximal set $S=\{ \beta_1, \ldots ,  \beta_r\}$ of strongly orthogonal roots spanning $\mathfrak n$. 
  (A root $\alpha$ is said to span $\mathfrak n$ if $\mathfrak g_{\alpha} \subseteq \mathfrak n$.) 
   A set $S$ can be constructed inductively as follows: $\beta_1$ is the highest root, 
 $\beta_2$ is the highest root amongst the roots spanning $\mathfrak n$  and orthogonal to $\beta_1$,  etc. For every $\beta_i\in S$ 
take an $\mathfrak {sl}_2$-triple $(f_i, h_i, e_i)$  where $e_i\in \mathfrak g_{\beta_i}$ and  $f_i\in \mathfrak g_{-\beta_i}$. 
We normalize the Killing form $\kappa(\cdot, \cdot )$ on $\mathfrak g$ by 
\[ 
\kappa(f_i, e_i )=1 
\] 
for all $i$.  Each triple $(f_i, h_i, e_i)$  lifts to a homomorphism of algebraic groups 
\[ 
\varphi_i : \SL_2 \rightarrow G. 
\] 

By restricting $\varphi_i$ to the torus of diagonal matrices in $\SL_2$ we obtain a homomorphism 
(a co-character)  $\chi^{\vee}_i : \mathbb G_m \rightarrow M$, 
\begin{equation}  \label{E:co-character} 
\chi^{\vee}_i(t) = \varphi_i \left(
\begin{matrix} t & 0 \\ 
0& t^{-1} \end{matrix}
\right). 
\end{equation} 
Let $T_r\subseteq M$ be the torus generated by all  $\chi_i^{\vee}(t)$.  
Any element in $T_r(k)$ is uniquely written as a product of $\chi_i^{\vee}(t_i)$ for some $t_i\in k^{\times}$. 
Let $\chi$ be a generator of the group of characters $\Hom(M, \mathbb G_m)\cong \mathbb Z$. The kernel of $\chi$ is $M_{\der}$, the derived group of $M$. 
From the root data it is easy to check that (for one of the two choices of $\chi$) 
\begin{equation} \label{E:co-root}
\chi(\chi^{\vee}_i(t))=t. 
\end{equation} 
 
\smallskip 

 Let 
\[ 
f=\sum_{i=1}^r f_i, \, h=\sum_{i=1}^r h_i \text{ and }  e=\sum_{i=1}^r e_i. 
\] 
Since the roots $\beta_i$ are strongly orthogonal, $(f,h,e)$ is also an $\mathfrak {sl}_2$-triple. 
The semi-simple element $h$ preserves the decomposition 
\[ 
\mathfrak g = \bar{\mathfrak n}\oplus \mathfrak m \oplus \mathfrak n. 
\] 
More precisely,  $[h,x]=-2 x$ for all $x\in \bar{\mathfrak n}$, $[h,x]=0$ for all $x\in {\mathfrak m}$, and 
$[h,x]=2 x$ for all $x\in \mathfrak n$. The triple $(f, h, e)$  lifts to a homomorphism 
\[ 
\varphi : \SL_2 \rightarrow G. 
\] 
The element  
\[ 
w=
 \varphi \left(
\begin{matrix} 0 & 1 \\ 
-1&  0 \end{matrix}
\right) 
\] 
normalizes $M$ and conjugates $\mathfrak n$ into $\bar{\mathfrak n}$, and vice versa. Explicitly, the action of $w$ on $x\in \mathfrak n$ is given by 
\[ 
w(x) = \frac{1}{2} [f,[f,x]]. 
\]

\smallskip 

\subsection{Jordan algebras.} Using the $\mathfrak{sl}_2$-triple $(f,h,e)$ 
we can define  a Jordan algebra  structure on $J=\mathfrak n$ with  multiplication $ \circ$  
\begin{equation} \label{E:jordan}
x\circ y = \frac{1}{2}[ x, [f,y]]. 
\end{equation} 
Note that $e$ is the identity element.  Similarly, we can define a Jordan algebra structure  on $\bar{\mathfrak n}$ with the multiplication $ \circ$  
\[ 
x\circ y = \frac{1}{2}[ x, [-e,y]].
\] 
In this case $-f$ is the identity element. 
These two structures are isomorphic under the conjugation by $w$. We shall now discuss this structure in more details, working with 
$\mathfrak n$.

The elements $e_i$ are mutually perpendicular ($e_i \circ e_j =0$ if $i\neq j$) and idempotent ($e_i\circ e_i=e_i$) elements in $J$ such that 
$e_1+ \cdots + e_r=e$. These idempotent elements give a Pierce decomposition of $J$, 
\[ 
J = \bigoplus_{1\leq i \leq r} J_{ii} \oplus \bigoplus_{1\leq i<j\leq r} J_{ij} 
\] 
where   
\[ 
J_{ii} = \{ x\in J ~|~ e_i \circ x = x\}
\] 
and 
\[ 
J_{ij} = \{ x\in J ~|~ e_i \circ x =\frac{1}{2}  x \text{ and } e_j \circ x= \frac{1}{2} x \}. 
\] 
The space $J_{ii}$ is one-dimensional and spanned by $e_i$. The space $J_{ij}$  can also be described in terms of the original root data. It is a span 
of $\mathfrak g_{\alpha} $ such that  
\begin{equation} \label{E:pierce}
\langle \alpha, \beta_i^{\vee}\rangle = \langle \alpha, \beta_j^{\vee}\rangle=1 \text{ and }   \langle \alpha, \beta_l^{\vee}\rangle=0 \text{ if } l\neq i,j. 
\end{equation} 
Since the Weyl group of $M$ can be used to reorder the elements of $S$ in any way (see \cite{rrs}), 
all $J_{ij}$  have the 
same  dimension $d$, as in the table. With respect to the conjugation action of $M$ on $N$, 
$\chi_i^{\vee}(t)$ acts on $J_{ii}$ by multiplication by $t^2$, on $J_{ij}$ by multiplication by $t$, 
and trivially on all other summands in the Pierce decomposition of $J$.

\begin{prop} \label{P:composition} Let $\kappa$ be the Killing form on $\mathfrak g$, normalized so that $\kappa(f_i, e_i )=1$ for all $i$. 
For every pair of  indices $i\neq j$ let $Q_{ij}$ be a quadratic form on $J_{ij}$ defined by   
\[ 
Q_{ij}(x)= \frac{1}{2}\kappa([f_i, x], [f_j,x]]). 
\] 
Then,  for  every $x\in J_{ij}$, 
\[ 
x\circ x = Q_{ij}(x) (e_i+e_j). 
\] 
The  quadratic form $Q_{ij}$ is non-degenerate and split, that is, it contains a direct sum of $[d/2]$ hyperbolic planes.  Let 
\[ 
\{ x,y\} = [x, [f,y]] =2 (x\circ y) 
\] 
denote the ``Jordan bracket''.  The quadratic forms $Q_{ij}$ satisfy a composition property: Let  $i,j,l$ be three distinct indices.  
For every $x\in J_{il}$ and $y\in J_{ij}$, so $\{ x,y\}\in J_{jl}$, 
\[ 
 Q_{jl}(\{x\circ y\})= Q_{il}(x) \cdot Q_{ij}(y). 
\]  
\end{prop} 
\begin{proof} We first show that $\{x, y\}$, for $x,y\in J_{ij}$, is a multiple of $e_i+e_j$.  Since $J_{ij}$ is a span of $\mathfrak g_{\alpha}$ satisfying  (\ref{E:pierce}), 
$[f_l,y]=0$ for all $l\neq i,j$. Hence 
\[ 
\{x, y\}= [ x, [f,y]] = [ x, [f_i,y]] + [ x, [f_j,y]]. 
\] 
Exploiting (\ref{E:pierce}) again, $[ x, [f_i,y]]$ is contained in a sum of $\mathfrak g_{\alpha}$ such that $\langle \alpha , \beta_j^{\vee}\rangle=2$. But  this equation holds only
for $\alpha=\beta_j$. Hence $[ x, [f_i,y]]$ is a multiple of $e_j$, while $[ x, [f_j,y]]$ is a multiple of $e_i$. In order to determine the 
coefficient in front of $e_j$  we take the inner product of $[ x, [f_i,y]]$ and  $f_j$,  with respect to the Killing form.  By the invariance of the Killing form, we have 
\[ 
\kappa(f_j, [ x, [f_i,y]])= \kappa([f_i, x], [f_j,y]])=  \kappa(f_i, [ x, [f_j,y]]). 
\] 
This proves that $ x\circ x = Q_{ij}(x) (e_i+e_j)$, as claimed. 
We go on to describe  the structure of the quadratic form $Q_{ij}$. On the set of roots $\alpha$ spanning $J_{ij}$  we have an involution 
\[ 
\alpha \mapsto \alpha^* = \beta_i+ \beta_j -\alpha. 
\] 
If $\alpha$ is fixed by the involution then $2\alpha = \beta_i+\beta_j$. 
This is only possible in the cases $C_n$ and $B_{n+1}$. In both cases there is only one fixed root, a short root. 
The complement of this line (if there is such a line) is a sum of hyperbolic planes, 
 $\mathfrak g_{\alpha} \oplus \mathfrak g_{\alpha^*}$ for  $\alpha \neq \alpha^*$. This completes the proof of the first part of the proposition. The second part, 
 the composition property of quadratic forms $Q_{ij}$, follows from a beautiful but long computation that we omit. 
\end{proof}

In order to describe the algebra $J$ we need to review some facts from the theory of Jordan algebras. 
A Jordan algebra $J$ has degree $r$ if  any element $x$ in $J$ satisfies a generic minimal polynomial 
\[ 
x^r - a_{r-1}x^{r-1} + \cdots  + (-1)^r a_0=0 
\]
where $a_i \in k$ depend algebraically on $x$. The coefficients  $a_{r-1}$ and $a_0$ are the trace $T_J$ and the norm $N_J$ of $x$. 

Let $D$ be a composition algebra over $k$.  It is a unital, not necessarily associative, algebra with a non-degenerate quadratic form $N_D$ (the norm) such that 
$N_D(uv)=N_D(u) N_D(v)$. The possible dimension of $D$ are $1,2,4$ or $8$. There is a linear map $u\mapsto \bar u$ on $D$ such that 
$\overline{uv}=\bar v\bar u$  and $N_D(u)=u\bar u$, for all $u,v\in D$.  Let $T_D(u) = u+\bar u$. It is a linear functional, called the trace.   
  We shall consider the following three families of Jordan algebras in this paper. 

\vskip 5pt 
\noindent 
\underline{Special Jordan algebras.} Assume that $D$ is associative,
 {\it{i.e.}} $\dim  D\neq 8$.
Let $H_r(D)$ be the set of hermitian-symmetric $r\times r$ matrices $x$ with entries in 
$D$, 
{\it{i.e.}} any element in $H_r(D)$ is equal to its transpose-conjugate, where by conjugation we mean applying the map $u\mapsto \bar u$ to all entries. 
If $\dim D\neq 8$, then $H_r(D)$ is a Jordan algebra with respect to the operation 
\[ 
x\circ y = \frac{1}{2} (xy+yx) 
\] 
where $xy$ and $yx$ are the usual multiplication of $r\times r$ matrices.  The norm $N_J$ is the reduced determinant. 

\vskip 5pt 
\noindent 
\underline{Exceptional Jordan algebras ($r=3$).} 
Assume that $\dim D=8$.  
Then $H_3(D)$ is a Jordan algebra only for $r=3$.  The norm $N_J$ of  
\[ x=\left(\begin{matrix} 
a & u & \bar w \\ 
\bar u & b & v \\
w & \bar v & c \\ 
\end{matrix} \right)
\] 
is 
\[ 
N_J(x) = abc -aN_D(v) - bN_D(w) -cN_D(u) + T_D(vwu). 
\] 
\vskip 5pt 
\noindent 
\underline{Quadratic Jordan algebras ($r=2$).} 
Let $(V, Q)$ be a non-degenerate quadratic space ever $k$, where $V$ is a vector space and $Q$ is a non-degenerate
 quadratic form on $V$. Let 
 \[ 
 J_2(V)=J_2(V,Q)=k e_1 \oplus ke_2 \oplus V.
 \] 
 In particular, an element in $J_2(V)$ is a triple $(a,b,v)$ where $a,b\in k$ and $v\in V$. The Jordan square in $J_2(V)$ is defined by 
 \[ 
 (a,b,v) \circ (a,b,v)= (a^2+ Q(v), b^2+ Q(v), av+bv). 
 \] 
  Then $e_1$ and $e_2$ are orthogonal idempotents  such that $e=e_1+e_2$  is the identity in $J_2(V)$. The norm $N_J$ is 
  \[ 
  N_J(a,b, v)= ab +Q(v). 
  \] 
 
 \vskip 5pt

 \begin{prop}  
If the type of $G$ is $A_{2n-1}$, $D_{2n}$ or $E_7$, and $r\geq 3$,  then  $J$ is isomorphic to $H_r(D)$ where $D$ is a split composition 
algebra of dimension $d =2,4$ or $8$, respectively. If the type of $G$ is $D_{n+1}$ or $B_{n+1}$, the cases when $r=2$, then 
 $J$ is isomorphic to $J_2(V)$ where $V=J_{12}$ with the quadratic form  $Q_{12}$.  \end{prop} 
\begin{proof} 
If the type of $G$ is $A_{2n-1}$, $D_{2n}$ or $E_7$, then the forms $Q_{ij}$ are isotropic.
In particular, for every $i=2, \cdots, r$,  there exists $u_{1i}\in J_{1i}$ such that 
$Q_{1i}(u_{1i})=1$.  
Let $u_{ij}= \{ u_{1i}, u_{1j}\}$.  Then, by Proposition \ref{P:composition}, $Q_{ij}(u_{ij}) =1$. Hence $u_{ij}\circ u_{ij}= e_i+e_j$ for all pairs $i\neq j$.  
We  define a product  $\cdot$ on $J_{12}$ by 
\[ 
x\cdot y=\{\{x, u_{23}\}, \{ y, u_{13}\} \}. 
\] 
The composition property of quadratic forms $Q_{ij}$, as in Proposition \ref{P:composition},  implies that 
\[ 
Q_{12} (x\cdot y) = Q_{12}(x) Q_{12}(y)
\] 
making $J_{12}$ a composition algebra $D$, with the identity $1_D=u_{12}$. 
 Let $E_{ij}$ denote the elementary $r \times r$ matrix, all entries 0 except $(i,j)$ where the entry is 1.  
By Jacobson's coordinatization theorem
 \cite[page 101]{MC}, 
there is an isomorphism 
$J \overset \sim \rightarrow H_r(D)$ defined by 
\[ 
e_i \mapsto E_{ii}, \, u_{ij} \mapsto E_{ij}+ E_{ji} , \text{ and } v\mapsto v E_{12} + \bar{v} E_{21}, \, v\in D. 
\] 
 
In the last two cases, $D_{n+1}$ and $B_{n+1}$,   the algebra $J$ is obviously isomorphic to $J_2(J_{12})$. 
\end{proof} 

The conditions of Jacobson's coordinatization theorem can be always satisfied by picking $f_i$, $i=2, \ldots, r$,  suitably. Indeed, rescaling 
$f_i$ amounts to rescaling $Q_{1i}$. In particular, we can easily arrange that all $Q_{1i}$ represent $1$. For example, if $G=\Sp_{2n}(k)$, then we can 
arrange $J\cong H_n(k)$.  We fix, henceforth, the identification of $J$ with $H_r(D)$ or $J_2(V)$.  

\section{Kantor--Koecher--Tits construction} \label{S:KKT}

We continue with the assumptions and notations from the previous section. In particular, $\mathfrak g$ is split. 
Recall that we have an isomorphism of $\mathfrak n$ and $\bar{\mathfrak n}$, preserving the Jordan algebra structure $J$,  given by 
\[ 
{\mathfrak {n}} \to \bar{\mathfrak {n}}, 
\qquad x \mapsto w(x) = \frac{1}{2}[f,[f,x]].   
\] 
Let $C$ be the centralizer of the triple $(f,h,e)$ in $\Aut(\mathfrak g)$. Note that $C$ acts naturally on both 
$\mathfrak n$ and $\bar{\mathfrak n}$, preserving the Jordan algebra structure $J$. In this way we have a natural homomorphism 
\[ 
\iota : C \rightarrow \Aut(J).
\] 
\begin{prop} The map $\iota$ is an isomorphism of the centralizer $C$ in $\Aut(\mathfrak g)$ of the $\mathfrak{sl}_2$-triple $(f,h,e)$ and $\Aut(J)$, the automorphism group of $J$. 
\end{prop} 
\begin{proof} The proof is based on the following two lemmas.

\begin{lemma} \label{L:first}  We have $[\mathfrak n, \bar{\mathfrak n}]=\mathfrak m$. 
\end{lemma} 
\begin{proof}  Since $h=[e,f]$ and $h$ spans a complement of $\mathfrak m_{\der}$ in $\mathfrak m$, it remains to show that  
$\mathfrak m_{\der} \subseteq [\mathfrak n, \bar{\mathfrak n}]$. The algebra $\mathfrak m_{\der}$ 
is spanned by the $\mathfrak{sl}_2$-triples $(f_{\alpha}, h_{\alpha}, e_{\alpha})$, where $\alpha$  is a root in $\mathfrak m$. Now 
observe that any root $\alpha$  in $\mathfrak m$ is a sum of a root  $\gamma$  in $\mathfrak n$ and a root $\bar{\gamma}$  in $\bar{\mathfrak n}$. 
Hence $e_{\alpha}$ and $f_{\alpha}$, non-zero multiples of $[e_{\gamma}, e_{\bar{\gamma}}]$ and  $[f_{\gamma}, f_{\bar{\gamma}}]$ respectively, are contained in 
$[\mathfrak n, \bar{\mathfrak n}]$. Since $h_{\alpha}$ is a linear combination of $h_{\gamma}$ and 
$h_{\bar{\gamma}}$, it is also contained in $[\mathfrak n, \bar{\mathfrak n}]$. 
\end{proof} 

If $c\in C$ is in the kernel of $\iota$ then $c$ acts trivially on $\mathfrak n$ and $\bar{\mathfrak n}$. Since $c$ is an automorphism of $\mathfrak g$, it also acts trivially on 
$[\mathfrak n, \bar{\mathfrak n}]=\mathfrak m$. Hence $c=1$ and $\iota$ is injective. We now go on to prove surjectivity of $\iota$. 
Let $g\in \Aut(J)$. It acts naturally on 
$\mathfrak n$ and on $\bar{\mathfrak n}$. The two actions are related by the isomorphism $w$, that is,  $g(w(x))=w(gx)$ for every $x\in \mathfrak n$. We shall see that this 
action extends, uniquely, to an automorphism of $\mathfrak g$ fixing the triple  $(f,h,e)$.  
Uniqueness is clear. Indeed, by Lemma \ref{L:first}, any element in $\mathfrak m$ is a equal to a sum $\sum [x, w(y)]$, where 
$x, y \in \mathfrak n$,   hence $g$ must act on it by 
\begin{equation} \label{E:equation}
g([x, w(y)])=\sum [gx, w(gy)] 
\end{equation} 
in order to preserve the Lie algebra structure on $\mathfrak g$.  
However, it is not clear that this defines an action of $g$ on  $\mathfrak m$ since an element in $\mathfrak m$ can be written as a sum of the brackets 
in more than one way. To address this issue we need the following beautiful lemma that expresses the Lie bracket $[\mathfrak m, \mathfrak n]$ in terms of the Jordan algebra structure. 

\begin{lemma}  \label{L:second}  Let $x,y, z \in \mathfrak n$. Then
 \[ 
 [[x,w(y)], z]=2( (x\circ z)\circ y -  (z\circ y)\circ x - (x\circ y)\circ z)
 \]  
 where the left-hand side is computed using the Lie bracket, 
while the right-hand side is computed using the Jordan multiplication $\circ$ on $\mathfrak n$. 
\end{lemma} 
\begin{proof} This follows by substituting $w(y)=\frac{1}{2}[f,[f,y]]$, using the Jacobi identity, and the definition of the Jordan multiplication $\circ$ on $\mathfrak n$. 
\end{proof} 

 If $\sum [x, w(y)]= \sum [u,w(v)] \in \mathfrak m$ then 
\[ 
\sum [[x, w(y)],z] = \sum [[u,w(v)],z] \in \mathfrak n 
\] 
for all $z\in \mathfrak n$. Acting by $g$ on both sides of this equation, applying the second lemma, and using that $g$ is an automorphism of $J$, we have 
\[ 
\sum [[gx, w(gy)], gz] = \sum[ [gu, w(gv)],gz]
\] 
for  all $z\in \mathfrak n$. Since $\mathfrak m$ acts faithfully on $\mathfrak n$, 
it follows that $\sum [gx, w(gy)]= \sum [gu,w(gv)]$. 
Hence the action of $g$ on $\mathfrak m$ given by the equation (\ref{E:equation}) is well-defined.  

Lemma \ref{L:second}  (and an analogue of this lemma for the bracket $[\mathfrak m, \bar{\mathfrak n}]$) imply that $g$, acting on $\mathfrak g$, preserves the Lie bracket.  Since $g$ fixes $e$ and $f$, it fixes 
$h=[e,f]$. Thus $g$ is in $C$. This proves that $\iota$ is surjective. 

\end{proof}

Thus we have a natural map 
\[ 
H^1(k, \Aut(J) ) \rightarrow H^1(k, \Aut(\mathfrak g)). 
\] 
In particular, a class $c$ in $H^1(k, \Aut(J) )$ gives a form $J^c$ of $J$, a form $\mathfrak g^c$ of $\mathfrak g$, and a form $G^c$ of $G$. Since $c$ fixes 
the triple $(f,h,e)$, the triple is contained in $\mathfrak g^c$ and $w$ in $G^c$.  The adjoint action of $h$ on $\mathfrak g^c$ gives a decomposition 
\[ 
\mathfrak g^c = \bar{\mathfrak n}^c\oplus \mathfrak m^c \oplus \mathfrak n^c 
\] 
and $\mathfrak n^c$,  with the multiplication given by the equation (\ref{E:jordan}), is the Jordan algebra $J^c$. 
On the level of Lie algebras, this is the Kantor--Koecher--Tits construction. Moreover, the group $G^c$ can be related to Koecher's construction 
\cite{ko}. Koecher  considers the group generated by the birational transformations of $J^c$:  translations $t_y(x)=y+x$, for every $y\in J^c$,  and 
$j(x)=-x^{-1}$. Note that $N^c w P^c/ P^c$ is an open set in the Grassmannian $G^c/P^c$. The natural action of $G^c$ on $G^c/P^c$ by left translations 
gives  a group of birational transformations of $N^c$ where the action of $y\in N^c$ on $N^c$ is by $t_y$, while the action of $w$  on $N^c$ is by $j$. 
In particular, the group defined by Koecher is the adjoint quotient of $G^c$.

\subsection{Our groups} \label{SS:our_groups} 
In this paper we shall consider the groups $G^c$ where the cocycle $c$  arises as follows: 
 If $J=H_r(D)$ then there is a natural map $\Aut(D) \rightarrow \Aut(J)$. If $J= J_2(V)$ then 
 there is a natural map $\Aut(V) \rightarrow \Aut(J)$, where $\Aut(V)$ is the group of automorphisms of the quadratic space $(V,Q)$. 
 We shall assume that $c$ lies in the image of $H^1(k, \Aut(D))$ or $H^1(k, \Aut(V))$, respectively.  In particular the resulting Jordan algebra $J^c$ is 
 isomorphic to $H_r(D^c)$ or $J_2(V^c)$, respectively. All triples $(f_i, h_i, e_i)$, $i=1, \ldots, r$,  are contained in $\mathfrak g^c$, and the torus $T_r$ is contained in $G^c$. 
 The restricted root system with respect to $T_r$ is of the type $C_r$. 
\smallskip

\section{Hasse principle} \label{S:hasse} 

Let $G$ be constructed by means of a Jordan algebra $J=H_r(D)$ or $J_2(V)$, as in Section \ref{SS:our_groups}.  Thus, $D$ is any composition algebra and 
$V$ any non-degenerate quadratic space over $k$. 
In particular, we have a maximal parabolic subgroup $P=MN$ such that $N$ has a structure of the Jordan algebra $J$. 
 To be precise, $\mathfrak n$ carries  a Jordan algebra structure, however, $\mathfrak n$ is canonically isomorphic to $N$, hence $N$ carries the same Jordan algebra structure. Also, by an abuse of notation, we shall view  $e_i\in \mathfrak n$ as elements of $N$. 
A purpose of this section is to prove a Hasse principle for  $M$-orbits on $N$.  As $N$ and $\bar N$ are conjugate by the element  $w\in G$ 
preserving the Jordan structures and normalizing $M$, describing  $M$-orbits on $N$ is  equivalent to describing 
$M$-orbits on $\bar N$.  For notational convenience we work with $N$. First, we have  the following (see \cite{rrs} and \cite{sw}): 

\begin{prop} \label{P:hasse_split} 
Assume that $G$ is split. If the type of $G$ is $C_n$, in addition,  assume that $k$ is algebraically closed. 
Then  every $M_{\der}$-orbit on $N$ contains precisely one of the following: 
$e_1 + \cdots + e_j $, for some $j<r$, or $e_1+ \cdots +  e_{r-1} + ae_r$, for some $a\in k^{\times}$. 
\end{prop} 

In general, when $G$ is not necessarily split but  $J=H_r(D)$ or $J_2(V)$,  then we have a decomposition 
\[ 
N =\coprod_{j=0}^r  \Omega_{j}
\] 
where, for $j<r$,  $\Omega_{j}$ is the set of elements in $N$ in the orbit of $e_1+ \cdots + e_j$ over the algebraic closure. Informally speaking, 
$\Omega_j$ consist of elements of rank $j$. For example, if
 $J=H_r(D)$  where $D$ is an associative division algebra, then $\Omega_j$ consists of all matrices of rank $j$.

In general, $\Omega_j$ consists of possibly infinitely many $M$-orbits. 
We shall now work towards a description of $M$-orbits. The adjoint action of the torus $T_r$ on $\mathfrak g$ and 
$\mathfrak m$  gives rise to 
 (restricted) root systems of type $C_r$ and $A_{r-1}$, respectively. Let $\{ \epsilon_i -\epsilon_j ~|~ 1\leq i\neq j \leq r\}$ be the standard realization of the root system $A_{r-1}$. Then, for every  root $\epsilon_i -\epsilon_j$ there is a unipotent group 
 \[ 
 X_{ij} \subset M_{\der}
 \] 
 isomorphic to $D$, if $J=H_r(D)$, or to $V$, if $J=J_2(V)$. We shall describe $X_{ij}$ on a case by case basis. 
 \vskip 5pt 
\noindent 
\underline{$J=H_r(D)$ and $\dim D\neq 8$.} In this case $M_{\der}=\SL_r(D)$. Let $u\in D$. Let $x_{ij}(u)$ be an $r\times r$ matrix with $1$ on the diagonal, $u$ as 
$(i,j)$-entry and 0 elsewhere. Then $X_{ij}$ is the set of all $x_{ij}(u)$. Note that $x_{ij}(u)$ acts on $x\in H_r(D)$ by 
\[ 
x_{ij}(u) x x_{ji}(\bar u). 
\] 
 
 \vskip 5pt 
\noindent 
\underline{$J=H_3(D)$ and $\dim D= 8$.} In this case $M_{\der}$ is the group of linear transformations of $J$ preserving the norm 
$N_J$.  Let $u\in D$. Let $x_{ij}(u)$ be a $3\times 3$ matrix with $1$ on the diagonal, $u$ as 
$(i,j)$-entry and 0 elsewhere. Although $D$ is not associative, it is still true that 
\[ 
(x_{ij}(u) x) x_{ji}(\bar u)=x_{ij}(u) (x x_{ji}(\bar u )), 
\] 
for every $x\in H_3(D)$.  The group $X_{ij}$ is the set of linear transformations of $H_3(D)$ defined by 
\[ 
x\mapsto x_{ij}(u) x x_{ji}(\bar u). 
\] 

\vskip 5pt 
\noindent 
\underline{$J=J_2(V,Q)$.}  In this case $M_{\der}=\Spin(J)$ where $J$ is considered a quadratic space with respect to the norm $N_J$. 
Let $B(u,v)$ be the symmetric  bilinear form such that 
$B(v,v)=2 Q(v)$. The group $X_{ij}$  consists of elements $x_{ij}(u)$, 
$u\in V$, acting on $J$ by  
\[ 
x_{12}(u) (a,b,v)= (a, b + a Q(u) +B(u,v), v+ au) 
\] 
and 
\[ 
x_{21}(u) (a,b,v)= (a+b Q(u)+ B(u,v), b , v+ bu). 
\] 
 
\vskip 5pt 
Now, using the action of $X_{ij}$, 
 it is a simple exercise to check that any $M_{\der}$-orbit
 in $J$ contains  
\[ 
x=a_1 e_1 + \cdots + a_r e_r 
\]
 for some $a_1, \ldots , a_r \in k$.  
 If $a_r=0$ then  $\chi_r^{\vee}(t)$, defined by the equation (\ref{E:co-character}), stabilize $x$. Since $\chi(\chi_r^{\vee}(t))=t$, where 
 $\chi$ is the generator of $\Hom(M, \mathbb G_m)$, it readily follows that the $M_{\der}$-orbit of $x$ coincides with the $M$-orbit of $x$. 
 Hence, $M$-orbits and 
  $M_{\der}$-orbits in $\Omega_j$ coincide for all $j<r$. This observation will prove useful in the proof of the following Hasse principle. 
  
  \begin{thm} \label{T:hasse} 
   Let $k$ be a number field. Let $x, y\in \Omega_j(k)$ where $j<r$. If $x, y$ belong to the same $M(k_v)$-orbit for all 
  places $v$ of $k$, then $x,y$ belong to the same $M(k)$-orbit. 
  \end{thm} 
  \begin{proof} We shall prove this statement for $M_{\der}(k)$. 
  If $G$ is split but not of the type $C_n$, then there is nothing to prove, in view of  Proposition \ref{P:hasse_split}. Now assume that $J=H_r(D)$ where $D$ is an associative 
  division algebra over $k$. In this case $M_{\der}(k)=\SL_r(D)$  and $x,y$ can be viewed as hermitian forms on $D^r$.
   If two $k$-rational hermitian forms are equivalent over $k_v$, for all places $v$, then they are equivalent over $k$. 
   This is the classical weak local to global principle, see Chapter 10 in \cite{sch}. Of course, the equivalence refers to the action of $\GL_r(D)$, however, for degenerate forms 
  $\GL_r(D)$-equivalence is the same as $\SL_r(D)$-equivalence. Hence the Hasse principle holds in this case. 
  
We shall study the remaining  cases using Galois cohomology. 
  Let $C$ be the stabilizer of $e_1$ in $M_{\der}$, in the sense of algebraic groups. Then $M_{\der}(k)$-orbits in $\Omega_1(k)$ correspond to the elements in the kernel of the morphism 
\[ 
H^1(k, C) \rightarrow H^1(k, M_{\der}) 
\] 
of pointed sets.  Recall that $N$ is an irreducible representation of $M_{\der}$ and $e_1$ is the highest weight vector of weight $\beta$. 
Hence the stabilizer in $M_{\der}$ of the line through $e_1$ is a parabolic subgroup $LU$ such that the simple roots of the Levi factor $L$ are the simple roots of 
$M_{\der}$ perpendicular to  $\beta$. 
 If  the type of $G$ is not $C_n$ or $A_{2n-1}$  then $\beta$ is a fundamental weight for $M_{\der}$. Thus, in these cases, the stabilizer $C$ of $e_1$  is $L_{\der} U$. 
 Since  $H^1(k, L_{\der}U)= H^1(k, L_{\der})$ (the Galois cohomology of the unipotent group $U$ is trivial)  
  $M_{\der}(k)$-orbits in $\Omega_1(k)$ correspond to the elements in the kernel of the morphism 
\[ 
H^1(k, L_{\der}) \rightarrow H^1(k, M_{\der}) 
\] 
of pointed sets. Let $S_{\infty}$ be the set of archimedean places for $k$.  Since $L_{\der}$ and $M_{\der}$ are simply connected, the natural maps 
\[ 
H^1(k, L_{\der}) \rightarrow \prod_{v\in S_{\infty}} H^1(k_v, L_{\der}) 
\] 
and 
\[ 
H^1(k, M_{\der}) \rightarrow \prod_{v\in S_{\infty}} H^1(k_v, M_{\der}) 
\] 
are bijections.  Thus, if $G$ is not $C_n$ or $A_{2n-1}$, the Hasse principle holds for $\Omega_1$. In fact, we have the following, more precise, information. 
\begin{itemize} 
\item $M_{\der}(k_v)$ acts transitively on $\Omega_1(k_v)$ if $v$ is a $p$-adic place.   
\item  the number of  $M_{\der}(k)$-orbits in $\Omega_1(k)$ is equal to the product of the number of $M_{\der}(k_v)$-orbits in $\Omega_1(k_v)$ over all 
archimedean places $v$. 
\end{itemize} 

Finally, the case of $\Omega_2$ for $H_3(D)$, where $\dim D=8$. The stabilizer in $M_{\der}$ of $e_1+e_2$ is a connected group whose Levi 
factor is a simple, simply-connected  group of type $B_4$, see \cite{cc}. Hence the Hasse principle applies in this case, as well. 
\end{proof} 

\begin{cor} Assume that $k$ is a $p$-adic field. Then $M(k)$ acts transitively on $\Omega_1(k)$  unless $G$ has type $C_n$ or $A_{2n-1}$. In these two cases, when 
$J=H_n(D)$ and 
$\dim D=1$ or $2$,  then the orbits are parameterized by $k^{\times}/N_D(D^{\times})$. 
\end{cor} 
\begin{proof} 
Indeed, by the first bullet above, there is one orbit unless $G$ has type $C_n$ or $A_{2n-1}$. In these two cases, by looking at the explicit action of 
$\SL_n(D)$ on $\Omega_1$,  $t\cdot e_1$ and $u\cdot e_1$ are in the same orbit if and only if $t/u\in N_D(D^{\times})$.
\end{proof}

\section{Some preliminaries} \label{S:preliminaries} 

  Let $H$ be an algebraic group defined over $\mathbb R$.  
We shall write $H$ in place of $H(\mathbb R)$. We assume that $H$ is unimodular and fix an invariant Haar measure throughout this section. 
Take a faithful algebraic representation 
$\rho : H\rightarrow \SL_d(\mathbb R)$. 
Then any $g\in H$ is represented by a $d\times d$-matrix $(x_{ij})$. 
We set 
\[ 
||g||:= \sum_{ij} |x_{ij}|^2. 
\] 
A complex function $f$ on $H$ is called of {\it{moderate growth}}
 if there exists an integer $a$ such that 
$|f(g)|\cdot ||g||^a$ is a bounded function on $H$. On the other hand, a complex function $f$ on $H$ is called {\it{rapidly decreasing}}
 if,  for every integer $a$, 
$|f(g)|\cdot ||g||^a$ is a bounded function on $H$.

Let $\mathfrak h$ be the Lie algebra of $H$. 
Every element $u$ in $U(\mathfrak h)$,  the enveloping algebra of $\mathfrak h$, defines a left $H$-invariant 
differential operator acting on smooth functions. Let  $u\cdot f$ denote this action, where $f$ is a smooth function on $H$. 
The Schwartz space $\mathcal S(H)$ is the space of smooth functions  $f$ on $H$  such that 
$u\cdot f$ is rapidly decreasing for all $u\in U(\mathfrak h)$.

\subsection{Fr{\'e}chet spaces} A Fr{\'e}chet vector space over $\mathbb C$ is a complete locally convex vector space $V$ equipped with a countable family of semi-norms $|\cdot |_i$, $i\in \mathbb N$. 
The space $V$ is metrizable, 
 namely,
 it is homeomorphic
 to a complete metric space,
 {\it{e.g.}}
 with respect to the metric defined by 
\[ 
d(x,y)= \sum_{i=1}^{\infty} \frac{1}{2^i} \cdot \frac{|x-y|_i}{1+|x-y|_i}. 
\] 
Now it is not to difficult to see that a sequence $(x_i)$ in $V$ is Cauchy
 if and only it is so for every semi-norm. 

For a representation $\pi$ of $H$ on a Fr{\'e}chet space, we shall always assume the following. 
For every $v\in V$ the map $G \to V$, $g\mapsto \pi(g) v$ is continuous. 
For every $v\in V$ and any semi-norm $|\cdot |_i$  the function  
$g\mapsto |\pi(g)v|_i$ is of moderate growth.

A prominent example arises as follows. Let $\pi$ be a unitary representation of $H$ on a Hilbert space $\mathcal H$, 
with the invariant product $(\cdot , \cdot)_{\mathcal H}$, and the corresponding norm $|| \cdot ||$. 
Let $\mathcal H^{\infty}$ be the space of all smooth vectors in $\mathcal H$. A  vector $v$ of $V$ is smooth
 if the map $G \to V$, 
 $g\mapsto \pi(g)v$ is smooth, 
 or equivalently, 
for every $w\in \mathcal H$, $g\mapsto ( \pi(g)v, w)_{\mathcal H}$ is a smooth function. Then $\mathcal H^{\infty}$ is a Fr{\'e}chet space
 with respect to a family of the semi-norms 
\[ 
|v|_u=||d\pi(u) v|| 
\] 
 for every $u\in U(\mathfrak h)$, the enveloping algebra of $\mathfrak h$. 
 
\subsection{Integration} 
Let $\pi$ be a representation of $H$ on a Fr{\'e}chet space $V$. Then for every  continuous, rapidly decreasing function $\alpha$ on $H$
 we define an operator 
\[ 
\pi(\alpha) : V \rightarrow V 
\] 
by 
\[ 
\pi(\alpha) v= \int_H \alpha(x) \pi(x) v ~dx. 
\] 
For our working purposes,  $\pi(\alpha)v$ can be defined as the limit, in $V$, of a sequence of finite sums, as follows. 
For every $a \in {\mathbb{N}}$, 
one can take a sequence of finite sets
$X_a \subset H$,
  and  for every $x \in X_a$
 a measurable set $S_x^a$
 containing $x$
 such that $\| g_1^{-1} g _2\| \le 2^{-a}$
 for any $g_1$, $g_2 \in S_x^a$
 and such that for every 
continuous, rapidly decreasing function $\alpha$ the sequence 
\[ 
\sum_{x\in X_a} \mu_x \alpha(x) 
\] 
converges to the integral $\int_H \alpha(x)d x$
 where $\mu_x \equiv \mu_x^a = \int_{S_x^a}d x$
 ($< \infty$).  
Then, for every $v\in V$,  $\pi(\alpha)v$ is defined as the limit of the sequence 
\[ 
v_a= \sum_{x\in X_a} \mu_x \alpha(x) \pi(x) v.  
\] 

  For the sake of completeness, 
we make this precise in the case $H=\mathbb R$, essentially the only case that we shall use in this paper. 
 For every  $a\in \mathbb N$, 
 we take $x_i=2^{-a}i-2^{a-1}$
 ($0 \le i \le 4^a$) 
 and divide the interval $[-2^{a-1},2^{a-1}]$ into subintervals  $[x_i, x_{i+1}]$ of lengths $1/2^a$. 
  Let $X_a=\{x_1, \ldots x_{4^a}\}$ and $S_{x_i}^a=[x_i, x_{i+1}]$. 
   
\begin{lemma} For every $v\in V$, the sequence 
 \[ 
 v_a= \frac{1}{2^a}\sum_{i=1}^{4^a} \alpha(x_i)\pi(x_i) v,  \, a\in \mathbb N, 
 \] 
  is Cauchy with respect to any semi-norm $|\cdot |$ defining the topology of $V$. 
\end{lemma} 
\begin{proof} Let $A>0$ and, for every $a$,  write $v_a=v_a^{<A}+ v_a^{\geq A}$ where $v_a^{<A}$ is the sum over $x_i$ such that $\|x_i\| <A$.  
Since $\alpha(x)$ is rapidly decreasing
 and $|\pi(x)v|$ is of moderate growth, 
$|\alpha(x)\pi(x)v|$ is rapidly decreasing.  
Therefore, 
 given $\epsilon >0$,
 one can take $A$ large enough so that $| v_a^{\geq A}| < \epsilon/3$ for all $a$. Using the continuity of $\pi$, one shows that 
 \[ 
 | v_a^{< A}- v_b^{<A}| < \epsilon/3
 \] 
 for any $a,b$ large enough. Thus $|v_a- v_b| < \epsilon$ for all $a,b$ large enough. 
  \end{proof}

\begin{prop} \label{P:integration} 
 Let $\chi : H\rightarrow  \mathbb C^{\times}$ be a unitary character of $H$, 
 and $\ell : V \rightarrow \mathbb C$ a continuous functional such 
that $\ell(\pi(g) v) = \chi(g) \ell(v)$ for all choices of data. Then, for every $\alpha \in \mathcal S(H)$ and every $v\in V$, 
\[ 
\ell(\pi(\alpha) v) = \ell(v) \hat{\alpha}(\chi) 
\] 
where $\hat{\alpha}(\chi):=  \int_H \alpha(x)  \chi(x) ~ dx$. 

\end{prop} 
\begin{proof} Write $\pi(\alpha)v$ as the limit of $v_a =\sum_{x\in X_a} \mu_x \alpha(x) \pi(x) v$
 as $a$ tends to infinity. 
Since $\ell$ is assumed to be continuous, 
\[ 
\ell(\pi(\alpha)v)=\ell(\lim_{a\rightarrow \infty}v_a )=\lim_{a\rightarrow \infty} \sum_{x\in X_a} \mu_x \alpha(x) \ell(\pi(x) v) = 
\] 
\[ 
=\lim_{a\rightarrow \infty} \sum_{x\in X_a} \mu_x \alpha(x)\chi(x)  \ell( v) =\ell(v)   \int_H \alpha(x)  \chi(x) ~ dx
 = \ell(v) \hat \alpha (x). 
\] 

\end{proof} 

\subsection{$p$-adic case} Assume now that $k$ is a $p$-adic field. In this case $\mathcal S(H)$ is the space of locally constant, compactly supported functions on $H$. 
If $(\pi, V)$ is a smooth representation of $H$, then the operator $\pi(\alpha)$ 
is defined by 
\[ 
\pi(\alpha) v= \int_H \alpha(x) \pi(x) v ~dx
\] 
where, in this case, the right-hand side is a finite sum. In particular, the analogue of Proposition \ref{P:integration} trivially holds true. 
 
 \subsection{Fourier Transform} Assume now that $k$ is a local field. 
Let $N$ be the abelian unipotent radical of a maximal parabolic subgroup of $G$, as in 
 Section \ref{S:KKT}. 
Let $\psi$ be a unitary, non-trivial character of $k$. 
The Killing form  $\kappa$  defines a pairing between $N$ and $\bar{N}$ by 
\begin{equation} \label{E:pairing} 
\langle n, x\rangle = \kappa(\log n, \log x).
\end{equation}  
 In particular, every $x\in \bar N$ defines a 
unitary character of $N$ by 
\[ 
\psi_x(n)=\psi(\langle n,x \rangle). 
\]
 Let $\mathcal  S(N)$ be the space of Schwartz functions on $N$.  In this situation, we have a Fourier transform 
 $\mathcal  F : \mathcal  S(N) \rightarrow \mathcal  S(\bar{N})$ by
\[
\calF(\alpha )(x) = \hat{\alpha }(x) = \int_N \alpha (n) \psi_x(n) ~dn
\]
where $dn$ is a Haar measure on $N$.  It is well-known that the Fourier transform is a bijection
 between the two Schwartz spaces
 ${\mathcal {S}}(N)$ and $\mathcal  S(\bar{N})$. We shall need this fact. 

\section{An analogue of the Sobolev lemma}

Let $\pi$ be a unitary representation of $G$
 on a Hilbert space ${\mathcal{H}}$, 
and ${\mathcal{H}}^{\infty}$
 the space of smooth vectors.  
Let ${\mathcal{H}}^{-\infty}$ be 
 the set of distribution vectors
 consisting of linear functional
 on the Fr{\'e}chet space ${\mathcal{H}}^{\infty}$.  
We write
\[
\langle \hphantom{i}, \hphantom{i} \rangle:
{\mathcal{H}}^{\infty} \times {\mathcal{H}}^{-\infty} \to {\mathbb{C}}
\]
for the natural bilinear map.  
Then the Lie algebra ${\mathfrak {g}}$ acts
 on ${\mathcal{H}}^{-\infty}$
 as a contragredient representation:
 for $X \in {\mathfrak {g}}$, 
\[
 \langle w , d \pi^{-\infty}(X)v\rangle
:=
-\langle d \pi(X)w, v \rangle
\quad
\text{for }
w \in {\mathcal{H}}^{\infty}
\text{ and }
v \in {\mathcal{H}}^{-\infty}.  
\]
We extend $d \pi^{-\infty}$ 
 to a ${\mathbb{C}}$-algebra homomorphism
$
   U({\mathfrak {g}}) \to \operatorname{End}_{\mathbb{C}}({\mathcal{H}}^{-\infty}).  
$

Since ${\mathcal{H}}$ is a Hilbert space
 with inner product $(\,,\,)_{{\mathcal{H}}}$, 
 we may regard $v \in {\mathcal{H}}$
 as a distribution vector
 by 
\[
   \langle w, v \rangle :=(w,v)_{{\mathcal{H}}}
\quad
\text{for }w \in {\mathcal{H}}^{\infty}.  
\]
This yields an (anti-linear) embedding
\begin{equation}
\label{eqn:HHinfty}
{\mathcal{H}} \hookrightarrow 
{\mathcal{H}}^{-\infty}, 
\end{equation}
so that we have a Gelfand triple
$
{\mathcal{H}}^{\infty}
\subset
 {\mathcal{H}}
\subset 
{\mathcal{H}}^{-\infty}
$.

In general,
 for $v \in {\mathcal{H}}$
 and $u \in U({\mathfrak {g}})$, 
 $d \pi^{-\infty}(u)v$ is defined just as a distribution vector.  
However,
 if $d \pi^{-\infty}(u)v$ belongs
 to the Hilbert space ${\mathcal{H}}$
 which is identified as a subspace of ${\mathcal{H}}^{-\infty}$ 
 by \eqref{eqn:HHinfty}, 
 we get a better regularity on $v$.  
Here is an analogue of the Sobolev lemma 
 which we need:
\begin{prop}
\label{prop:Sobolev}
Suppose $v \in {\mathcal{H}}$ satisfies
\[
   d \pi^{-\infty}(u)v \in {\mathcal{H}}
\qquad
\text{for all } 
u \in U({\mathfrak {g}}).  
\]
Then $v$ is a smooth vector.  
\end{prop}
This proposition is a consequence of iterated applications
 of the following lemma:
\begin{lemma}
\label{lem:Xf}
Let $X \in {\mathfrak {g}}$.  
Suppose $v \in {\mathcal{H}}$ satisfies
\[
d \pi^{-\infty}(X)v \in {\mathcal{H}}
\quad
\text{and }
\quad
d \pi^{-\infty}(X^2)v \in {\mathcal{H}}.  
\]
Then 
$
  \lim_{t \to 0}
  \frac 1 t (\pi(e^{tX})v-v)
$
 converges to $d \pi^{-\infty}(X)v$
 in the topology of the Hilbert space ${\mathcal {H}}$.  
\end{lemma}

\begin{proof}
Take any $w \in {\mathcal{H}}^{\infty}$, 
and we set
\[
  f(t):=(w, \pi(e^{tX})v)_{{\mathcal{H}}}
       =(\pi(e^{-tX})w,v)_{{\mathcal{H}}}.  
\]
Since $w$ is a smooth vector,
 $f(t)$ is a $C^{\infty}$-function on ${\mathbb{R}}$.  
By Taylor's theorem
 there exists $0 < \theta<1$
 such that 
\[
   f(t)=f(0)+t f'(0) +\frac{t^2}{2} f''(\theta t), 
\]
where 
\begin{align*}
f(0)=& (w,v)_{{\mathcal{H}}}, 
\\
f'(0)=&(-d \pi (X)w,v)_{{\mathcal{H}}}=\langle w, d \pi^{-\infty}(X)v\rangle, 
\\
f''(s)=&(d \pi (X) d \pi (X) \pi (e^{sX})w,v)_{{\mathcal{H}}}
      =\langle \pi (e^{sX})w, d \pi^{-\infty}(X^2)v\rangle.  
\end{align*}
Since $d \pi^{-\infty}(X)v \in {\mathcal{H}}$, 
we have 
$
  f'(0)= (w, d \pi^{-\infty} (X)v)_{{\mathcal{H}}}.  
$
Since $d \pi^{-\infty}(X^2)v \in {\mathcal{H}}$, 
 and since $\pi$ is a unitary representation,
 the remainder term has an upper estimate 
\[
 |f''(s)| \le \|w\|_{{\mathcal{H}}} 
              \| d \pi^{-\infty}(X^2)v \|_{{\mathcal{H}}}.  
\]
Thus we have
\[
\left| (w,\frac{\pi(e^{tX})v-v}{t}-d \pi^{-\infty}(X)v)_{{\mathcal{H}}} \right|
\le
\frac{|t|}{2}
\|w\|_{{\mathcal{H}}} 
\| d \pi^{-\infty}(X^2)v \|_{{\mathcal{H}}}.  
\]
Since ${\mathcal{H}}^{\infty}$ is dense in ${\mathcal{H}}$, 
 the above estimate holds
 for all $w \in {\mathcal{H}}$.  
Hence we have shown the lemma.  
\end{proof}

\section{Small representations of $p$-adic groups} 

Assume that $k$ is a $p$-adic field. 
Let $G$ be a group defined over $k$,  as in Section \ref{S:KKT}. 
In particular, we have a maximal parabolic subgroup $P=MN$ with abelian unipotent radical $N$. 
Fix a non-trivial character $\psi$ of $k$. 
Then every  $y\in \bar N$ defines a unitary character $\psi_y$ of $N$, $\psi_y(n)=\psi(\langle n, y\rangle)$, 
where $\langle n, y \rangle$ is the pairing between $N$ and $\bar N$ defined in (\ref{E:pairing}).
Fix an $M$-orbit  $\omega \subseteq \bar N$.  We shall consider $M$ acting on  $\omega$  from the left.  
Let $dx$ be a quasi $M$-invariant measure on $\omega$. 
Let $\nu : M \rightarrow \mathbb C^{\times}$ be a smooth character such that 
\[ 
d(mx)=|\nu(m)|^{-2} dx. 
\] 
On $L^2(\omega)$  we have a unitary, irreducible,  representation of $P$ where $m\in M$ and $n\in N$ act on $f\in L^2(\omega)$ by, respectively, 
\[ 
\pi(m)f(y)= \nu(m)f(m^{-1} y) 
\] 
and 
\[ 
\pi(n)f(y)= \psi_y(n)f(y). 
\] 

Assume that $\pi$ extends to a unitary representation of $G$. 
Let $V$ be the space of $G$-smooth elements in $L^2(\omega)$. Let $\ell$ be a functional on $V$ such that $\ell(\pi(n)v)= \psi_x(n) \ell(v)$ for all choices of data. The main goal of this section is to prove that $\ell=0$  if $x$ does not belong to the topological closure of $\omega$ and $\ell(f)= \lambda  f(x)$, for some $\lambda\in \mathbb C$, 
if $x$ belongs to $\omega$, 
 see Proposition \ref{P:p-adic}.

\begin{lemma} Every $M$-smooth element in $L^2(\omega)$ is represented, uniquely, by a locally constant function on $\omega$. 
\end{lemma} 
\begin{proof} Let $f$ be an $M$-smooth element, 
{\it{i.e.}}  there exists an open compact subgroup $K$ of $M$ fixing $f$, not as a function on $\omega$, but in the $L^2$-sense. Write 
\[ 
\omega=\coprod_{i} \omega_i 
\] 
where each $\omega_i$ is an $K$-orbit. It is an open compact subset of $\omega$. In particular, the restriction of $f$ to $\omega_i$ is well-defined. 
Let $f_i$ be that restriction. Then 
\[ 
f_i\in L^2(\omega_i)^K. 
\] 
Now recall that  $\dim L^2(\omega_i)^K = 1$ by computing the trace of the projection operator, for example. Hence $L^2(\omega_i)^K$ 
is spanned by the characteristic function of $\omega_i$, and $f_i$ is represented by a constant function on $\omega_i$. 
Therefore $f$ is represented by a locally constant function. 
The uniqueness is clear. 
\end{proof}

\begin{prop}  \label{P:p-adic} 
 Let $x\in \bar N$.  Let $\ell$ be a functional on  $V$ such that 
$\ell(\pi(n) f) =\psi_x(n) \ell(f)$ for all choices of data.  
\begin{itemize} 
\item If $x$ is not in the topological closure of $\omega$ then $\ell=0$. 
\item If $x\in \omega$, then there exists $\lambda\in \mathbb C$ such that $\ell(f) =\lambda f(x)$ for all $f$. 
\end{itemize}  
\end{prop} 
\begin{proof} 
Assume first that $x$ is not in the topological closure of $\omega$.  Let $B_x$ be an open neighborhood of  $x$ in $\bar N$ disjoint from  the topological closure of $\omega$. 
  Let $\alpha \in \mathcal S(N)$ be such that $\hat\alpha(x)=1$ and the support of $\hat\alpha$ is contained in $B_x$. 
Let $f\in V$. We shall  now compute $\ell(\pi(\alpha) f)$ in two ways. The first uses the explicit definition of $\pi$, 
\[ 
\pi(\alpha) (f)(y)= \int_N\alpha(n) \pi(n)(f)(y) ~dn =  \int_N\alpha(n) \psi_{y}(n)f(y) ~dn= \hat{\alpha}(y) f(y) =0 
\] 
since $\hat{\alpha}(y)=0$ for all $y\in \omega$. Hence $\ell(\pi(\alpha) f)=0$.

The second computation uses the formal property of $\ell$, as in Proposition \ref{P:integration}, 
\[ 
\ell(\pi(\alpha) f)=\ell(f)  \int_N \alpha(n) \psi_x(n) ~dn = \ell(f)  \hat{\alpha} (x)= \ell(f). 
\] 
Thus $\ell(f)=0$ for all $f\in V$, by combining the two computations. This proves the first bullet.

For the second, let $V_x\subseteq V$ be the subspace  
of codimension one consisting of $f$ such that $f(x)=0$. We need to show that $\ell(f)=0$ for all $f\in V_x$.
 Fix $f\in V_x$. Let $B_x$ be a neighborhood of $x$  in $\bar N$ such that  $f$ vanishes on $B_x\cap \omega$. 
Let  $\alpha$ be such that $\hat{\alpha}(x)=1$ and the support of  $\hat{\alpha}$ is in $B_x$. With this modifications, the argument used in the proof of the first bullet implies that $\ell(f)=0$. The proposition is proved. 
\end{proof} 

\section{Small representations of real groups} 
 
Let $k=\mathbb R$, and fix a character $\psi: \mathbb R \rightarrow \mathbb C^{\times}$, $\psi(z) =e^{\sqrt{-1}z}$. Then any $y\in \bar N$ defines a 
unitary character of $N$ by 
\[ 
\psi_y(n) = e^{\sqrt{-1} \langle n, y \rangle}  
\] 
where $\langle n, y \rangle$ is the pairing between $N$ and $\bar N$ defined in (\ref{E:pairing}).  Let $dx$ be a quasi $M$-invariant measure on $\omega$. 
Let $\nu : M \rightarrow \mathbb C^{\times}$ be a smooth character such that 
\[ 
d(mx)=|\nu(m)|^{-2} dx. 
\] 
Then, as in the $p$-adic case, we have an irreducible unitary representation $(\pi, \mathcal H)$ of $P$ where 
$\mathcal H =L^2(\omega)$ and  $m\in M$ and  $n\in N$ act on $f\in L^2(\omega)$ by, respectively, 
\[ 
\pi(m)(f)(y)=\nu(m) f(m^{-1}y) 
\] 
and 
\[ 
\pi(n)(f)(y)=\psi_y(n) f(y). 
\] 

Now assume that $\pi$ extends to a unitary representation of $G$. 
In particular, we assume that the $G$-invariant Hilbert space structure is given 
by the inner product $( \cdot, \cdot)_{\mathcal H}$ arising from the $L^2$-norm. Let $\mathcal H^{\infty}$ be the Fr{\'e}chet space of $G$-smooth vectors. 
Let $\ell$ be a continuous  functional on $\mathcal H^{\infty}$ such that $\ell(\pi(n)v)= \psi_x(n) \ell(v)$ for all choices of data. The main goal of this section is to prove Proposition \ref{P:real} asserting
 that 
$\ell=0$  if $x$ does not belong to the topological closure of $\omega$ and $\ell(f)= \lambda f(x)$, for some  $\lambda\in \mathbb C$, 
if $x$ belongs to $\omega$, under the following, natural, assumption on $\pi$. 

A regular differential operator $D$ on $\omega$ is called {\it{anti-symmetric}} if,  for any $\varphi\in C_c^{\infty}(\omega)$ and $f\in C^{\infty}(\omega)$, 
\[ 
\int_{\omega} D \varphi \cdot \bar f  = -\int_{\omega}  \varphi \cdot \overline{ D f}. 
\] 
Since $M$ acts transitively on $\omega$, $M$-smooth elements in $\mathcal H$ are represented by smooth functions on $\omega$, see \cite{Po}. 
In particular, all elements in $\mathcal H^{\infty}$ are represented by smooth functions on $\omega$. 
We assume that $\mathfrak g$ acts on $\mathcal H^{\infty}$ by 
anti-symmetric regular differential operators, that is,  for every $X\in \mathfrak g$ there exists an anti-symmetric regular differential operator $D_X$ such 
  that $d\pi(X) f = D_X f$ for all $f\in \mathcal H^{\infty}$. 

\vskip 15pt 
\begin{lemma} \label{L1:real}
 With the above assumptions, $C_c^{\infty}(\omega) \subseteq \mathcal H^{\infty}$. 
\end{lemma} 
\begin{proof} Let $\langle \cdot,\cdot \rangle$ be the natural pairing between $\mathcal H^{\infty}$ and $\mathcal H^{-\infty}$. 
 Let $\varphi \in C_c^{\infty}(\omega)$.  Then $\varphi$ is viewed as an element in  $\mathcal H^{-\infty}$ by 
\[ 
\langle f,  \varphi \rangle := ( f, \varphi)_{\mathcal H} 
\] 
 for all  $f\in \mathcal H^{\infty}$.  For every $X\in \mathfrak g$, let  $d\pi^{-\infty}(X)\varphi\in \mathcal H^{-\infty}$
  be the weak derivative of $\varphi\in \mathcal H^{-\infty}$, that is, 
\[ 
\langle f, d\pi^{-\infty} (X)\varphi \rangle : =  -(  d\pi(X)f,   \varphi )_{\mathcal H} 
\] 
for all $f\in \mathcal H^{\infty}$. Since, by the assumption,  $d\pi(X) f= D_X f$ for an anti-symmetric regular differential operator $D_X$, we have 
\[ 
\langle f, d\pi^{-\infty} (X)\varphi \rangle =  (  f,   D_X\varphi )_{\mathcal H}. 
\] 
It follows that all weak derivatives of $\varphi$ are contained in $\mathcal H$. Hence $\varphi \in \mathcal H^{\infty}$, by Proposition \ref{prop:Sobolev}. 
 
\end{proof}

 \begin{lemma} \label{L2:real}
 For every $f\in L^2(\omega)$ and $\alpha\in \mathcal S(N)$, 
$ \pi(\alpha) (f)= \hat{\alpha}  f$, the point-wise product of  $\hat{\alpha}$ and $f$. 
\end{lemma} 
\begin{proof}  Recall, from Section  \ref{S:preliminaries},   that $\pi(\alpha)f$ is defined as a limit, in $L^2(\omega)$,  of the sequence of finite sums 
 \[ 
 f_a= \sum_{x\in X_a}\mu_x  \alpha(x)  \pi(x) f 
 \] 
where $\sum_{x\in X_a}\mu_x  \beta(x)$ converges to $\int_N\beta$
  for every continuous, rapidly decreasing function  $\beta$ on  $N$. In particular,  for every $y$, the sequence 
\[ 
f_a(y)= \sum_{x\in X_a} \mu_x \alpha(x)  \pi(x)(f)(y)=  \sum_{x\in X_a} \mu_x \alpha(x) \psi_y(x)f(y)
\] 
converges  to $\hat{\alpha} (y)f(y)$. Thus the sequence of functions $f_a$ converges pointwisely
 to $\hat{\alpha}f$. 
In order to show that $f_a$ converges to $\hat{\alpha}f$ in the $L^2$-norm we shall apply Lebesgue's dominated convergence theorem.  Using the triangle inequality and $|\psi_x(n)|=1$, 
\[ 
|f_a(y) |\leq ( \sum_{x\in X_a} \mu_x |\alpha(x)|)  \cdot |f(y)|. 
\] 
Since $( \sum_{x\in X_a} \mu_x |\alpha(x)|)$ converges to $C=\int_N |\alpha(x)| d x$, it follows that  $|f_a(y)| \leq (C+1) |f(y)|$ for almost all $a$. Since 
$|\hat{\alpha}(y)|\leq C$ we also have $|(\hat{\alpha} f)(y)|\leq C|f(y)|$.  Hence  
\[ 
|(f_a -\hat{\alpha}f)(y)|^2\leq (2C+1)^2 |f(y)|^2. 
\] 
Thus, by Lebesgue's dominated convergence theorem, we can exchange the limit and integration in the following. 
\[ 
\lim_{a\rightarrow \infty}  \int_{\omega} |f_a-\hat{\alpha}f|^2= \int_{\omega}\lim_{a\rightarrow \infty} |f_a- \hat{\alpha}f|^2 =0. 
\] 
\end{proof} 

\begin{prop} \label{P:real} Assume that  for every $X\in \mathfrak g$ there exists an anti-symmetric regular differential operator $D_X$ such 
  that $d\pi(X) f = D_X f$ for all $f\in \mathcal H^{\infty}$. 
Let $x\in \bar N$. Let $\ell$ be a continuous functional on  $\mathcal H^{\infty}$ such that $\ell(\pi(n) f) =\psi_x(n) \ell(f)$ for all choices of data.   
\begin{itemize} 
\item If $x$ is not in the topological closure of $\omega$, then $\ell=0$. 
\item If $x\in \omega$, then there exists $\lambda \in \mathbb C$ such that $\ell(f) = \lambda f(x)$ for all $f\in \mathcal H^{\infty}$. 
\end{itemize} 
\end{prop} 
\begin{proof} Assume that $x$ is not in the topological closure of $\omega$.
 Let $B_x$ be an open neighborhood of  $x$  in $\bar N$ disjoint from the topological closure of $\omega$. 
  Let $\alpha \in \mathcal S(N)$ be such that $\hat\alpha(x)=1$ and the support of $\hat\alpha$ is contained in $B_x$. Then $\pi(\alpha)(f)=0$, for all $f$,  
  by Lemma \ref{L2:real}.  On the other hand, by Proposition \ref{P:integration}, 
 \[ 
 \ell(\pi(\alpha) f) = \hat{\alpha}(x) \ell(f)=\ell(f). 
 \] 
 Combining the two gives $\ell(f)=0$ for all $f$. This proves the first bullet. 

For the second bullet, note that the same argument proves that $\ell(f)=0$ for any function $f\in \mathcal H^{\infty}$ that vanishes in an open  neighborhood of $x$. 
 Let $d$ be the dimension of $\omega$. Since $\omega$ is a homogeneous space for $M$, it is a smooth manifold. Hence we can take $v_1, \ldots ,v_d \in N$ giving a local chart around $x$. More precisely,  every $y\in \omega$ close to $x$ is identified with a $d$-tuple of real numbers  $y_i=\langle v_i, y\rangle$, $i=1, \ldots, d$. 
 In particular,  $x$  is identified with the $d$-tuple of real numbers $x_i=\langle v_i, x\rangle$. 
Let $O_x$ be an open neighborhood of $x$ in $\omega$,  identified with 
\[ 
I= \{ (y_1, \ldots , y_d) \in \mathbb R^d ~|~ |y_i-x_i|< \epsilon \}
\] 
for some $\epsilon >0$. 
Since $\mathcal H^{\infty}$ contains $C_c^{\infty} (\omega)$, by Lemma \ref{L1:real}, every $f\in\mathcal  H^{\infty}$ can be written as $f=f_1+f_2$ where $f_1$ vanishes in a neighborhood of $x$ and $f_2$ has support 
contained in $O_x$. Since $\ell(f_1)=0$, it remains to understand the restriction of $\ell$ to functions supported in $O_x$.

Let $f\in C_c^{\infty}(I)$ and $f_a\in C_c^{\infty}(I)$, $a\in \mathbb N$, a sequence  of functions  supported in a compact set $C\subset I$, such 
\[ 
\lim_{a\rightarrow  \infty} \sup_{y\in I} |Df_a(y) -Df(y)| =0
\] 
for all partial derivatives $D$ in the variables $y_i$. Using the identification $C_c^{\infty}(I)\cong C_c^{\infty}(O_x)$, 
consider $f$ and $f_a$ as elements in $\mathcal H^{\infty}$. 
Then, since $\frak g$ acts as regular differential operators, the sequence $f_a$ converges to $f$ in the topology of $\mathcal H^{\infty}$. Hence, 
$\lim_{a\rightarrow \infty}\ell(f_a) =\ell(f)$. It follows that  $\ell$ defines a distribution on 
$C_c^{\infty}(I)$ supported at $0$. 
By the structural theory of distributions,
 every such distribution is a finite linear combination of partial derivatives of the delta function $\delta_x$.

Let $X_i=\log(v_i) \in \mathfrak n$ and $y\in \bar N$. Using the definition of the pairing $\langle \cdot, \cdot \rangle$ in (\ref{E:pairing}), 
 we have 
\[ 
\langle e^{tX_i}, y\rangle = \kappa(t X_i, \log y)= t\cdot \kappa( X_i, \log y) =t \langle v_i, y\rangle =ty_i. 
\] 
Thus 
\[ 
\psi_y(e^{tX_i})  =e^{\sqrt{-1} t y_i} .
\]

By the equivariance of $\ell$, for every $t\in \mathbb R$, 
\[ 
\ell(\pi(e^{tX_i}) f) =\psi_x(e^{tX_i})\cdot \ell(f) = e^{\sqrt{-1} t x_i}  \cdot  \ell(f). 
\] 
Since $\ell$ is a continuous functional, we can pass to the action $d\pi$ of the Lie algebra, that is, we can differentiate with respect to $t$. This gives 
\begin{equation} \label{E:real} 
\ell(d\pi(X_i) f) = \sqrt{-1}  x_i \cdot  \ell(f). 
\end{equation} 

On the other hand, 
\[ 
\pi(e^{tX_i}) (f)(y) = \psi_y(e^{tX_i}) f(y)= e^{\sqrt{-1} t y_i} f(y). 
\] 
By passing to the action of $d\pi$, 
\[ 
d\pi(X_i)(f)(y)=  \sqrt{-1} y_i\cdot f(y). 
\] 
Substituting into (\ref{E:real}) yields $\ell((y_i-x_i) f) =0$. Hence $\ell(P\cdot f)=0$ for all $f\in C_c^{\infty}(I)$ and 
all polynomials $P$ in $y_i$ vanishing at $x$. This implies that $\ell$ is a scalar multiple of $\delta_x$, as claimed. 

\end{proof} 

\section{$N$-rank}

Let $\Omega_j$ be the set of elements in $\bar N$ of rank $j$ as defined in Section  \ref{S:hasse}.  Note that 
 $\Omega_j$ is not empty by our assumption on the Jordan algebra. 
 Over a local field, the topological closure of  $\Omega_j$ is the union of  $\Omega_i$ with $i\leq j$.

\subsection{Local rank}  
Assume that $k$ is a local field. 
We shall define a notion of $N$-rank for any smooth representation $(\pi, V)$ of $N$. Recall that, 
if  $k$ is archimedean, $(\pi, V)$ is  smooth representation on a Fr{\'e}chet space. In this case $V^*$ is the space of continuous functionals on $V$. 
If $k$ is $p$-adic, $V^*$  is the space of all functionals on $V$. 
Let $x\in \bar N$. Recall that every $x$ defines a unitary character $\psi_x$ of $N$. 
 Let $(V^*)^{N,\psi_x}$ be the subspace of $V^*$ consisting of all $\ell$ such that 
\[ 
\ell(\pi(n)v)= \psi_x(n) \ell(v) 
\] 
for all choices of data. 

\begin{definition}\normalfont
Let $(\pi,V)$ be a smooth representation of $N$. The largest integer $j$ such that $(V^*)^{N,\psi_x}\neq \{0\}$, for some 
 $x\in \Omega_j$, is called the {\it{local $N$-rank}} of $V$. Let $\omega$ be
 an $M(k)$-orbit in $\Omega_j$. 
Suppose that the local rank of $V$ is $j$.  
We say $V$ {\it{has pure rank $j$ relative to}}
 $\omega$, 
  if $(V^*)^{N,\psi_x} =\{0\}$  for all $x\in \Omega_j\setminus \omega$.   
\end{definition}

\begin{prop} \label{P1} Assume that the rank of a smooth representation $V$ of $N$ is larger than $j$. 
Then there exists 
$\alpha\in \mathcal  S(N)$ such that the support of $\hat{\alpha}$ is disjoint from the topological closure of $\Omega_j$ and $\pi(\alpha) \neq 0$. 
\end{prop}

\begin{proof} By the assumption, there exists $x\in \bar N$, not contained in the topological closure of $ \Omega_j$, and a non-zero, continuous functional $\ell$ on $V$ such that 
$\ell(\pi(n)v)= \psi_x(n) \ell(v)$,  for all choices of data. 
Take $v\in V$ such that $\ell(v)\neq 0$. 
Clearly, we can take $\alpha\in \mathcal S(N)$ such that $\hat{\alpha}(x)=1$ and 
the support of $\hat{\alpha}$ is disjoint from the topological closure of $\Omega_j$. Then by Proposition \ref{P:integration} 
\[ 
\ell(\pi(\alpha) v) = \hat{\alpha}(x) \ell(v)= \ell(v) \neq 0. 
\] 
\end{proof}

\subsection{Automorphic representations}
Assume that $k$ is a number field. Let $k_{\infty}=k\otimes \bbR$, and $\hat k$ be the completion of $k$ 
with respect to all discrete valuations on $k$. Then $\bbA=k_{\infty} \times \hat k$ is the ring of adel\'es corresponding to $k$. 
Let $\mathfrak g$ be the Lie algebra of $G(k_{\infty})$, and $U(\mathfrak g)$ the corresponding enveloping algebra.

Let $\mathcal A$ be the space of functions $f$ on $G(\bbA)$ such that 
\begin{enumerate} 
\item $f$ is left $G(k)$-invariant. 
\item $f$ is right $K_f$-invariant, where $K_f$ is an open compact subgroup of $G(\hat k)$, depending on $f$. 
\item For every $\hat g\in G(\hat k)$, $g_{\infty} \mapsto f(g_{\infty}, \hat g)$ is a smooth function. In particular,  $U(\mathfrak g)$ acts on 
$f$ from the right by left invariant regular differential operators. 
\item The condition (3) assures that $u\in U(\mathfrak g)$ acts on $f$, $f\mapsto u\cdot f$,  
 by a left invariant regular differential operator. We assume that 
$f$ is annihilated by an ideal $I$ of finite index in $Z(\mathfrak g)$, the center of $U(\mathfrak g)$. 
\item $f$ is of uniform moderate growth. This means that there exists an integer $d$ such that for all $u\in U(\mathfrak g)$, the function 
$|u\cdot f(g)| \cdot ||g||^d$ is bounded on $G(k_{\infty})$. 
\end{enumerate} 

Our definition of moderate  growth appears to be slightly different from the one in the literature, where it is usually required that
 $g_{\infty} \mapsto f(g_{\infty}, \hat g)$ has moderate growth on $G(k_{\infty})$ for every $\hat g\in G(\hat k)$. However, since $G$ is simply connected, 
 $G(\hat k)= G(k) K_f$ by the strong approximation. Now it is easy to see that the two definitions are equivalent. Moreover, 
  $f$ can be viewed as a function on $ \Gamma\backslash G(k_{\infty})$ where $\Gamma= G(k) \cap G(k_{\infty} )\cdot K_f$

Fix  $\hat K$,  an open compact subgroup of $G(\hat k)$,  $I$ and $d$. Let $\mathcal A(\hat K, I, d)$ be the subspace of $\mathcal A$ consisting of 
$f$ right invariant by $\hat K$, annihilated by $I$ and of moderate growth controlled by $d$ as above. On $\mathcal A(\hat K, I, d)$ we have a family of 
semi-norms 
\[ 
\sup_{g\in G(k_{\infty})}|u\cdot f(g)| \cdot ||g||^d, 
\] 
one for every $u\in U(\mathfrak g)$. Then $\mathcal A(\hat K, I, d)$ is a Fr{\'e}chet space with a smooth $G(k_{\infty})$-action. 
The underlying $(\mathfrak g, K_{\infty})$-module consists of modular forms. It is of finite length, by an old result of Harish-Chandra. 

The group $G(\mathbb A)$ acts on $\mathcal A$ by right translations. We shall denote this action by $R$. 
An irreducible automorphic representation is a subspace $\pi \subseteq \mathcal A$ invariant under the action of $G(\mathbb A)$ and satisfying the following additional conditions. There is a smooth representation $\pi_{\infty}$ of $G(k_{\infty})$ on a Fr{\'e}chet space,  a smooth representation 
$\hat\pi$ of $G(\hat k)$, and a $G(\mathbb A)$-intertwining isomorphism 
\[ 
 T : \pi_{\infty} \otimes \hat \pi \rightarrow \pi \subset \mathcal A. 
 \] 
 Moreover,  for every open compact subgroup $\hat K$ of $G(\hat k)$, the map $T$ is continuous $G(k_{\infty})$-intertwining map from 
 $\pi_{\infty} \otimes \hat \pi^{\hat K}$ to $\mathcal A(\hat K, I, d)$, for some $d$ and $I$. (Note that the Fr{\'e}chet topology on $\pi_{\infty}$ induces a canonical one on 
   $\pi_{\infty} \otimes \hat \pi^{\hat K}$ since $\hat\pi^{\hat K}$ is finite dimensional.)  Finally, we remark that $\hat \pi$ is a restricted direct product
   $\hat\otimes \pi_v$  of smooth irreducible representations $\pi_v$ of $G(k_v)$ for every finite place $v$.

\subsection{Global rank}  Let $\pi$ an irreducible automorphic representation. 
Fix a character $\psi : k \backslash\bbA \rightarrow \bbC^\times$. For $x \in \bar{N}(k)$, we define
$\psi_x : N(k) \backslash N(\bbA) \rightarrow \bbC^\times$ by $\psi_x(n) = \psi(\langle n,x \rangle)$. Then $f\in \pi$ admits a Fourier expansion
\[
f(g) = \sum_{x \in \bar{N}(k)} f_x(g)
\]
where
\[
f_x(g) = \int_{N(k)\backslash N(\bbA)} f(ng)\bar \psi_x(n) dn.
\]
The functional 
\[ 
\ell_x : \pi \rightarrow \mathbb C
\] 
defined by $\ell_x(f)= f_x(1)$ for all $f\in \pi$  satisfies 
\[ 
\ell_x(R(n) f)= \psi_x(n)\ell_x(f) 
\] 
for all $n\in N(\mathbb A)$ and $f\in \pi$. It is useful to note, and easy to check, that $\ell_x=0$ implies $\ell_y=0$ for all $y$ in the $M(k)$-orbit of $x$.

\begin{definition}\normalfont
Let  $\pi$ be an irreducible automorphic representation.  The largest integer $j$ such that  $\ell_x\neq 0$ for some $x\in \Omega_j$ is called 
the {\it{global $N$-rank}} of $\pi$. 
 Let $\omega$ be an $M(k)$-orbit in $\Omega_j$.  
Suppose that the global rank of $\pi$ is $j$.  
We say $\pi$ has {\it{pure rank $j$ 
relative to}} $\omega$,  
if $\ell_x= 0$ for all $x\in \Omega_j\setminus \omega$.  
\end{definition} 

\smallskip

\begin{thm} \label{global} 
Let $\pi$ be an irreducible automorphic representation.   If the global $N$-rank of $\pi$ is $j$ then, for any 
place $v$, the local component $\pi_v$  of $\pi$ has the local $N$-rank $j$. 
\end{thm}

\begin{proof} We fix an isomorphism $T$ of $\pi$ with $\pi_{\infty}\otimes \hat\pi$. We shall prove that 
 $\pi_{\infty}$ has rank $j$. The proof of the statement for the components of $\hat\pi$ is similar and easier, since there are no topological considerations. We leave this out as an exercise. 
Let $x \in \Omega_j$ such that $f_x(1) \neq 0$ for some $f \in\pi$.  Let $\hat K$ be an open compact subgroup in $G(\hat k)$ such that $f$ is left invariant under $\hat K$. Then $f$ lies in the image of $\pi_{\infty} \otimes \hat\pi^{\hat K}$. 
The map  $f \mapsto f_x(1)$ is clearly continuous in the topology of $\mathcal A(\hat K, I, d)$. Hence, by composing it with $T$, it gives a continuous, non-zero, 
functional on $\pi_{\infty}\otimes \hat\pi^{\hat K}$, a finite multiple of $\pi_{\infty}$.
Hence the local $N$-rank  of $\pi_{\infty}$ is greater or equal to $j$.

It remains to show that the rank of $\pi_{\infty}$ is not greater than $j$.
 By Proposition~\ref{P1}, it suffices to show that
$\pi_{\infty}(\alpha) = 0$ for any $\alpha \in \mathcal S(N_{\infty})$ such that the Fourier transform $\hat\alpha$ is supported on elements of rank $>j$. 
By using the intertwining map $T$, it suffices to prove that 
\[ 
R(\alpha)(f)=0
\] 
 for all $f\in T(\pi_{\infty}\otimes \pi^{\hat K})$, for some $\hat K$, 
 where $R$ denotes the representation of $G(k_{\infty})$ on 
$\mathcal A(\hat K, I, d)$, acting by right translations.  

\begin{lemma} Let $f\in A(\hat K, I, d)$, and $\alpha\in \mathcal S(N_{\infty})$. Then 
\[ 
R(\alpha)(f)(g)=\int_N f(gn) \alpha(n) dn. 
\] 
\end{lemma} 
\begin{proof} 
Recall that the operator $R(\alpha)(f)$ is defined as a limit, 
in the Fr{\'e}chet topology on $\mathcal A(\hat K, I, d)$, of a sequence of functions $f_a$, $a\in \mathbb N$, 
\[ 
f_a(g)= \sum_{n\in X_a} \mu_n  f(gn) \alpha(n) 
\] 
where, $X_a$ are finite sets in $N_{\infty}$ and $\mu_n$ positive real numbers such that for every continuous, rapidly decreasing function $\beta$ on $N_{\infty}$, the sequence $\sum_{n\in X_a} \mu_n\beta(n)$ converges to the 
integral of $\beta$. 

The topology of $\mathcal A(\hat K, I, d)$  is given by sup-norms, hence the convergence of $f_a$ implies  the convergence of $f_a(g)$ 
for every $g\in G(\mathbb A)$. Since, for every $g$, the function $n\mapsto f(gn)\alpha(n)$ is rapidly decreasing on $N_{\infty}$,  the sequence $f_a(g)$  converges to the integral of $f(gn) \alpha(n)$.  This proves the lemma. 

\end{proof} 

Since $R(\alpha)(f)$ is smooth function on $G(k)\backslash G(\mathbb A)$ and 
 $G(k)$ is dense in
$G(k_{\infty})$ (see Proposition 7.11 in \cite{PR}), it suffices to prove that  $R(\alpha)(f)=0$ on $G(\hat k)$. 
 Let $\hat g \in G(\hat k)$.  Firstly, we expand $R(\alpha)(f)(\hat g)$ using the Fourier series: 
\[
R(\alpha) f(\hat g) = \sum_{x \in \bar{N}(k)} (R(\alpha) f)_x(\hat g). 
\] 
We shall now analyze each individual summand. Using the Fubini Theorem, one easily justifies that   
\[ 
(R(\alpha) f)_x(\hat g) = \int_{N_{\infty}} \alpha(n)  f_x(\hat gn)~dn. 
 \] 
 Now observe that $\hat g$ commutes with $n\in N_{\infty}$ and that $f_x(n\hat g)=\psi_x(n)f_x(\hat g)$. Hence 
 \[ 
  \int_{N_{\infty}} \alpha(n)  f_x(\hat gn)~dn= 
 \int_{N_{\infty}} \alpha(n)  \psi_x(n) f_x(\hat g)~dn= 
 \hat{\alpha}(x) f_x(\hat g). 
\]
The last term is clearly 0. Indeed, 
 $\hat\alpha(x)=0$ if the rank of $x$ is $> j$ and $f_x=0$ otherwise, by the assumption on $f$. This proves the theorem.
 \end{proof}

\section{Global uniqueness of small representations} 

 Let $G$ be as in Section \ref{SS:our_groups}, defined over a number field  $k$. 
Let $\pi$ be a smooth irreducible representation of $G(\bbA)$. The multiplicity  $m(\pi)$ of 
$\pi$ in $\mathcal A$, the space of automorphic functions, is defined as 
  \[ 
  m(\pi) = \dim \Hom_{G(\bbA)}(\pi, \mathcal A).
  \]  

 We are now ready to prove that  automorphic representations  whose local components are small representations have multiplicity one. 
  The proof is analogous to the proof of multiplicity one for irreducible cuspidal automorphic representations of 
$\GL_n$ (see \cite{PS}), based on uniqueness of Whittaker functionals, which we briefly sketch.  Let $\mathcal A_0 \subseteq \mathcal A$ be the subspace of 
cuspidal automorphic forms for $\GL_n$ and  let $T_1$ and $T_2$ be two non-zero elements in 
$\Hom_{\GL_n(\bbA)}(\pi, \mathcal A)$. The uniqueness of local Whittaker functionals can be exploited to show that there exists non-zero 
complex numbers $c_1$ and $c_2$ such that $(c_1 T_1 + c_2 T_2)(\pi)$ is not generic, i.e. has no global Whittaker functional. Since any non-zero 
cuspidal representation of $\GL_n(\mathbb A)$ is generic, it follows that $c_1 T_1 + c_2 T_2=0$. 

\begin{thm}   \label{T:unique} 
Let $\pi =\hat{\otimes} \pi_v$ be a  smooth irreducible representation of $G(\bbA)$. 
For every place $v$,  assume that the representation $\pi_v$  has the $N$-rank $j<r$,  pure relative to a single $M(k_v)$-orbit $\omega_v$ in 
$\Omega_j(k_v)$, and 
\[
(\pi_v^*)^{N(k_v), \psi_x}\cong \mathbb C
\] 
 for $x\in \omega_v$. Then $m(\pi)\leq 1$. 
\end{thm} 
\begin{proof} Let $T\in \Hom_{G(\bbA)}(\pi, \mathcal A)$, $T\neq 0$.  The purity of $\pi_v$ and  the Hasse principle for $\Omega_j$, Theorem \ref{T:hasse},  imply that 
 $T(\pi)$, if non-zero,  is pure relative to a single $M(k)$-orbit $\omega$ in $ \Omega_j(k)$. Fix $x\in \omega$. 
For every $v\in\pi$, let 
\[ 
\ell_{x,T}(v)=f_x(1) 
\] 
  where $f=T(v)$ and $f_x(1)$ is the Fourier coefficient of $f$.  Then $\ell_{x,T}$ is a functional on $\pi$ such that 
$\ell_{x,T}(\pi(n)v)=\psi_x(n)\ell_{x,T}(v)$ for all $v$.  
If $T_1,T_2 \in \Hom_{G(\bbA)}(\pi, \mathcal A)$ and are non-zero  then, 
by the uniqueness of the functional at every place, there exist $c_1,c_2\in \mathbb C^{\times}$ such that 
$c_1 \ell_{x,T_1} + c_2 \ell_{x,T_2}=0$. Since 
 \[ 
c_1 \ell_{x,T_1} + c_2 \ell_{x,T_2}=  \ell_{x,c_1 T_1 + c_2 T_2},  
\] 
it follows that  $\ell_{x,c_1 T_1 + c_2 T_2}=0$. However, for any $T\in \Hom_{G(\bbA)}(\pi, \mathcal A)$,  $\ell_{x,T}=0$ for one $x\in \omega$ implies 
 $\ell_{y,T}=0$ for all $y\in \omega$. Hence 
 $(c_1 T_1 + c_2 T_2)(\pi)$ has the global rank strictly less than $j$.  
In turn, 
 Theorem \ref{global} implies that the local components of $(c_1 T_1 + c_2 T_2)(\pi)$ have the rank strictly less 
than $j$.  This is only possible if  $c_1 T_1 + c_2 T_2=0$.  Hence $\Hom_{G(\bbA)}(\pi, \mathcal A)$ is at most one dimensional. 
\end{proof} 

\smallskip 
We now look at the minimal representations.  A representation  of a real groups is minimal if the annihilator in $U(\mathfrak g)$ is the Joseph ideal. 
For the groups considered in this paper, Theorems A and B in \cite{hkm} imply that the minimal representations satisfy the conditions of Proposition \ref{P:real}. 
In turn, Proposition \ref{P:real} implies that the minimal representations satisfy the conditions of Theorem \ref{T:unique}.  On the other hand, 
a representation of a $p$-adic group is minimal 
if its character, viewed as a distribution around $0\in \mathfrak g$, is equal to   
\[ 
 \int_{O} \hat f  + c\hat{f}(0)  
\] 
where $\hat f$ is the Fourier transform of  $f\in \mathcal S(\frak g)$, and $O$ is a minimal $G$-orbit in $\mathfrak g$. 
(See \cite{MW} and \cite{gs} for more details.)  
For the groups considered in this paper, the minimal representations, when restricted to $P$, have a realization on $L^2(\omega)$  
where $\omega= \bar{\mathfrak n} \cap O$, see \cite{to}.  Now Proposition \ref{P:p-adic}  implies that the minimal representations satisfy the assumptions of Theorem \ref{T:unique}. 
Summarizing, we have the following corollary to Theorem \ref{T:unique}. 
(As conjectured in the introduction of \cite{ms}.) 

\begin{cor} \label{C:min} 
Let $\pi=\hat{\otimes} \pi_v$ be a smooth irreducible representation of $G(\mathbb A)$ such that any local component $\pi_v$ is minimal. Then $m(\pi)\leq 1$.  
\end{cor} 

.

\section{Acknowledgments}

This work was completed while Gordan Savin  was visiting the University of Tokyo in July 2014. 
Gordan Savin would like to thank the colleagues, staff and students at the University of Tokyo for excellent working conditions and stimulating atmosphere, and 
Dipendra Prasad for a remark regarding the Kantor--Koecher--Tits construction. 
The work of Gordan Savin has been partially supported by an NSF grant DMS-1359774. 
Toshiyuki Kobayashi has been partially supported
 by Grant-in-Aid for Scientific Research (A)(25247006) JSPS, 
 the Institute for Mathematical Sciences (Singapore)
 and the Institut des Haute {\'E}tudes Scientifiques.


\begin{thebibliography}{HKM}


\bibitem[CC]{cc} A. M. Cohen and B. N. Cooperstein, 
{\em The 2-spaces of the standard $E_6(q)$-module,} 
Geometriae Dedicata {\bf 25}, (1988), 467--480. 

\bibitem[GS]{gs}  W.-T. Gan and G. Savin, {\em On minimal representations, definitions and properties,} Representation Theory, 
{\bf 9} (2005),  46--93. 

\bibitem[HKM]{hkm} J. Hilgert, T. Kobayashi and J. M\"ollers, {\em Minimal representations via 
Bessel operators,}  J. Math. Soc. Japan, {\bf{66}} (2014), 349--414. 


\bibitem[Ho]{Ho} R. Howe, {\em Automorphic forms of low rank,}  
in Lecture Notes in Mathematics, {\bf 880}, pp. 211--248, Springer, 1981. 

\bibitem[Ja]{ja} N. Jacobson, {\em Structure and Representations of Jordan Algebras,} Colloquium Publications, 
American Mathematical Society, 1968, 453 pp. 

\bibitem[Ko]{ko} M. Koecher, {\em \"Uber eine Gruppe von rationalen Abbildungen, }  Invent. Math. {\bf 3},  (1967), 136--171. 
  
\bibitem[KM]{km} T. Kobayashi and G. Mano, {\em The Schr{\"o}dinger model for the minimal 
representation of the indefinite orthogonal group $O(p, q)$,}  Mem. Amer. 
Math. Soc. {\bf{212}}, (2011) no. 1000, vi+132 pp. 


\bibitem[MC]{MC} K. McCrimmon, 
{A Taste of Jordan Algebras,} Universitext, Springer, 2004, 554 pp. 

\bibitem[MS]{ms} S. Miller and S. Sahi, {\em Fourier coefficients of automorphic forms, character variety orbits, and small representations,}
Journal of Number Theory, {\bf 132} (2012), 3070--3108. 


\bibitem[MW]{MW} 
C. Moeglin and J.-P. Waldspurger, 
{\em Mod\`eles de Whittaker d\'eg\'en\'er\'es pour des groupes $p$-adiques,}
Math. Z. {\bf 196} (1987), 427--452. 

\bibitem[PS]{PS} I. Piatetski-Shapiro, 
{\em Multiplicity one theorems,} in Proceedings of Symposia in Pure and Mathematics, {\bf 33}, Part 2, pp. 209-212, AMS, 1979.  

\bibitem[PR]{PR}  V. P. Platonov and A. Rapinchuk, {Algebraic groups and number theory,}  Pure and Applied 
Mathematics, {\bf{139}}, Academic Press Inc., Boston, Ma, 1994. 

\bibitem[Po]{Po} N. S. Poulsen, 
{\em On $C^{\infty}$-vectors and intertwining bilinear forms for representations of Lie groups,} 
J. Funct. Analysis, {\bf 9} (1972), 87--120. 

\bibitem[RRS]{rrs} R. Richardson, G. R\"ohrle and R. Steinberg, 
{\em Parabolic subgroups with abelian unipotent radical,} 
Invent. Math., {\bf 110}
(1992), 649--671. 

\bibitem[Sa]{Sa} S. Sahi, 
{\em Explicit Hilbert spaces for certain unipotent representations,} 
Invent. Math., {\bf 110} (1992), 409--418. 

\bibitem[SW]{sw}  G. Savin and M. Woodbury, {\em Structure of internal modules and a formula for the spherical
vector of minimal representations}, J. of Algebra,
{\bf 312} (2007), 755--772.

\bibitem[Sch]{sch} W. Scharlau, {\em Quadratic and Hermitian Forms,} 
Grundlehren der mathematischen Wissenschaften, 
270,
Springer-Verlag, 1985.  

\bibitem[To]{to} P. Torasso, {\em M\'ethode des orbites de Kirillov--Duflo et repr\'esentations minimales des groupes simples sur un corps local 
de caractristique nulle,} Duke Math. J., {\bf 90} (1997), 261--377. 

\bibitem[We]{mw} M. Weissman,  {\em The
 Fourier--Jacobi map and small representations,}
Representation Theory, {\bf 7} (2003), 275--299.


\end{thebibliography}
\end{document}